\documentclass{article}
\usepackage{amsthm}
\usepackage{amsfonts,authblk}
\usepackage{amsfonts}
\usepackage{amsmath,amssymb}
\usepackage{enumitem}
\usepackage[utf8]{inputenc}
\usepackage{tikz}
\usepackage{pgfplots}
\usepackage{graphicx}
\usepackage[english]{babel}
\pgfplotsset{compat=newest}
\usetikzlibrary{calc}

\newcommand{\bigo}[1]{\mathcal{O}\left(#1\right)}
\newcommand{\subscript}[2]{$#1 _ #2$}
\newcommand{\B}{\mathcal{ B}}
\newcommand{\K}{\mathcal{ K}}
\renewcommand{\H}{ H}
\newcommand{\R}{\mathbb{ R}}
\newcommand{\ri}{\mathcal{ R}}
\newcommand{\N}{\mathbb{ N}}

\renewcommand{\r}{R}
\newcommand{\norm}[1]{\left|\left|#1\right|\right|}
\newcommand{\embed}{\hookrightarrow}
\newcommand{\rhofem}{\rho_n^{H,h}}
\newcommand{\rhofemprev}{\rho_{n-1}^{H,h}}
\newcommand{\pifem}{\pi_n^H}
\newcommand{\pifemprev}{\pi_{n-1}^H}
\newcommand{\dealii}{\texttt{deal.ii}}
\newcommand{\mLL}{{L^2(\Omega;L^2(Y))}}
\newcommand{\MLL}{{L^2(\Omega)}}
\usepackage{booktabs}
 \newtheorem{lemma}{Lemma}
\newtheorem{proposition}{Proposition}
\newtheorem{definition}{Definition}
\newtheorem{theorem}{Theorem}

\newtheorem*{remark}{Remark}

\title{A semidiscrete Galerkin scheme for a two-scale coupled elliptic-parabolic system: well-posedness and convergence approximation rates}
\author[1]{Martin Lind}
\author[1]{Adrian Muntean}
\author[1]{Omar Richardson\thanks{email: omar.richardson@kau.se}}
\affil[1]{Department of Mathematics and Computer Science, Karlstad University, Sweden}

\begin{document}

\maketitle

\begin{abstract}
In this paper, we study the numerical approximation of a coupled system of elliptic-parabolic equations posed on two separated spatial scales. The model equations describe the interplay between macroscopic and microscopic pressures in an unsaturated heterogeneous medium with distributed microstructures as they often arise in modeling reactive flow in cementitious-based materials. Besides ensuring the well-posedness of our two-scale model, we design two-scale convergent numerical approximations and prove \textit{a priori} error estimates for the semidiscrete case. We complement our analysis with simulation results illustrating the expected behaviour of the system. 

  \textit{Keywords:} elliptic-parabolic system, weak solutions, Galerkin approximations, distributed microstructures, error analysis, macroscopic mesh refinement strategy

  \textit{MSC (2010):} 35K58, 65N30, 65N15

\end{abstract}
\section{Introduction}

This work is concerned with the design and approximation of systems of evolution equations posed on two distinct spatial scales. The systems we have in mind involve coupled partial differential equations (PDEs) that explicitly encode two-scale interactions via transmission boundary conditions as well as production terms; see e.g. the PDE structures entering double or dual porosity models, models with distributed microstructures, fissured-media equations, as well as general two-scale models. Such models arise as descriptions of reactive flow through geometrically-structured porous media.

If the geometry of the porous media has a dual porosity structure, and hence, characteristic scales can possibly be separated, then PDE models with distributed microstructures are in theory able to describe the relevant multiscale spatial interactions like those occurring in gas-liquid mixtures. Now, the challenge shifts from the multiscale modeling to the computer implementation of multiscale models. Consequently, in this work we concern ourselves with the two-scale computability issue -- complex systems of evolution equations acting on two spatial scales are notoriously hard to compute, especially if moving boundaries or stochastic dynamics are involved within e.g.~the distributed microstructures. Combined with the so-called curse of dimensionality, this results in a computational problem of very high complexity.

In this paper, we discuss the case of an elliptic-parabolic coupling. We consider a coupled system of partial differential equations connected to multiscale descriptions of the evolution of the pressure arising in a compressible air-liquid mixture that distributes over two spatial scales (one called macroscopic, and one microscopic). This situation arises, for instance, in cementitious materials within concrete members -- typical composite porous materials where the amount of displaceable liquid is low and is practically trapped in the internal structure of the porous medium. The derivation of our particular model originates from applying a formal two-scale homogenization to a particular scaling of the level set equation coupled with Stokes equations for fluid flow (see \cite{lic} for details).
Highlights of the more non-standard challenges of this system of equations are the two-scale coupling, the mismatch in structure between the two equations, i.e. the presence of a time derivative in the microscopic equation and its absence in the microscopic one, as well as the nonlinear right hand side in the macroscopic equations.
In order to tackle these challenges we perform most of our analysis on the finite element level, using techniques from e.g. \cite{larsson2008partial}, \cite{ciarlet02} and \cite{thomee1984}.

If we assume the interface between air and liquid to remain fixed for a reasonable time span, then using homogenization techniques for locally periodic microstructures (compare with e.g. \cite{chechkin98}) leads in suitable scaling regimes to a so-called \textit{two-pressure evolution systems}.
This system can be expressed as coupled elliptic-parabolic equations that describe the joint evolution in time $t\in (0,T)$ ($T<+\infty$) of a parameter-dependent \textit{microscopic pressure} $R\rho (t,x,y)$ (where $R$ represents the universal gas constant) evolving with respect to $y\in Y\subset \mathbb{R}^d$ for any given macroscopic spatial position $x\in \Omega$ and a \textit{macroscopic pressure} $\pi(t,x)$ with $x\in \Omega$ for any given $t$.
An illustration of the two-scale geometry we have in mind is depicted in Figure~\ref{fig:two_scale_structure}.

\begin{figure}
    \centering
    \begin{tikzpicture}[scale=0.4]
    \def\l{10};
    \def\r{1};
    \draw[black,fill=gray] plot [smooth cycle] coordinates {(\r*8, 11 + \r*2) (\r*9, 11 + \r*4) (\r*8, 11 + \r*7.5) (\r*5.5, 11 + \r*9) (\r*4, 11 + \r*8.5) (\r*2, 11 + \r*5)  (\r*2.5, 11 + \r*2) (\r*7, 11 + \r*1)};

    \draw [dashed, thick, step=\l] (0,0) grid (\l,\l);
    \draw [black] plot [smooth cycle] coordinates {(\l*0.2, \l*0.2) (\l*0.2, \l*0.6) (\l*0.5,\l*0.8) (\l*0.6, \l*0.7) (\l*0.85, \l*0.5) (\l*0.7, \l*0.3) (\l*0.6, \l*0.4)};
    \draw [black, fill=darkgray] plot [smooth cycle] coordinates {(\l*0.2, \l*0.2) (\l*0.2, \l*0.6)  (\l*0.85, \l*0.5) (\l*0.7, \l*0.3) (\l*0.6, \l*0.4)};
    \node at (0.5*\l,0.63*\l) (robin) {$\Gamma_R$};
    \node at (0.5*\l,0.30*\l) (neumann) {$\Gamma_N$};
    \node at (0.6*\l,0.30*\l+17.5) (macroboundary) {\large $\partial \Omega$};
    \node at (0.4*\l,0.45*\l) (micro) {\huge $Y$};
    \node at (0.6*\l,0.45*\l + 12) (macro) {\huge $\Omega$};
    \node at (\r*4+0.5, \r *15) (x) {\large $x$};
    \draw[->, thick,dashed] (x) -- (\r*4+0.5, \r *9);
\end{tikzpicture}
    \caption{The macroscopic domain $\Omega$ and microscopic pore $Y$ at $x\in\Omega$.}
    \label{fig:two_scale_structure}
\end{figure}
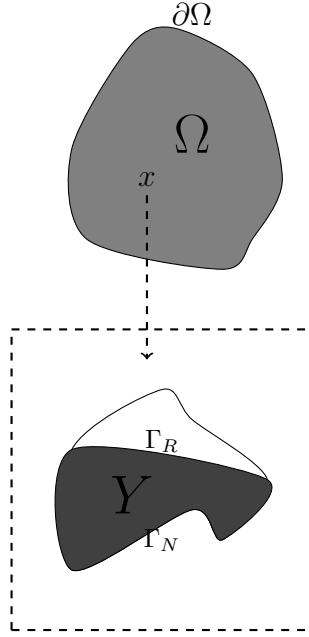

In the final part of this manuscript, we discuss a numerical implementation of this scheme in the finite element library \dealii{} (\cite{dealII}). Inspired by the Heterogeneous Multiscale Method framework (cf. e.g.\cite{engquist07}), we propose an implementation strategy that resolves the scale separation inherent in the two-scale structure of our problem.

We consider the following problem, posed on two spatial scales $\Omega\subset \R^{d_1}$ and $Y \subset \R^{d_2}$ with $d_1,d_2 \in \{1,2,3\}$ in the time interval $t\in S := (0,T)$ for some $T>0$. Find the two \emph{pressures} $\pi: S\times\Omega \to \R $ and $\rho: S\times\Omega\times Y\to \R$ that satisfy:

\begin{align}
    \label{eq:main_ellip}&-A\Delta_x\pi=f(\pi,g(\rho))  &\mbox{ in }S\times\Omega,\\
    \label{eq:main_para}&\partial_t\rho-D\Delta_y\rho = 0  &\mbox{ in }S\times\Omega\times Y,\\
    \label{eq:main_robin}&D\nabla_y\rho\cdot n_y= k(\pi+p_F-R\rho)&\mbox{ in } S\times\Omega\times\Gamma_R,\\
    \label{eq:main_neumann}&D\nabla_y\rho\cdot n_y=0&\mbox{ in }S\times\Omega\times\Gamma_N,\\
    \label{eq:main_dirichlet}&\pi=0 &\mbox{ in }S\times\partial\Omega,\\
    \label{eq:main_initial}&\rho(t=0)=\rho_I &\mbox{ in } \overline{\Omega\times Y},
\end{align}
where $\Gamma_R \cup \Gamma_N = \partial Y$, $\Gamma_R \cap \Gamma_N = \emptyset$ and $f,g$ are functions discussed in more detail below. We refer to \eqref{eq:main_ellip}-\eqref{eq:main_initial} as ($P_1$).

Note that $(P_1)$ describes the interaction between a compressible viscous fluid (with density $\rho$) in a porous domain $\Omega$, where the pores are partly filled with a gas that exerts an average (macroscopic) pressure $\pi$.
The interaction between the fluid and the gas is determined by the right hand side of \eqref{eq:main_ellip} and the microscopic boundary condition in \eqref{eq:main_robin}, through the fluid-gas interface represented by $\Gamma_R$.
The mathematical problem stated in ($P_1$), contains a number of dimensional constant parameters: $A$ (gas permeability), $D$ (diffusion coefficient for the gaseous species),  $p_F$ (atmospheric pressure) and $\rho_F$ (gas density).
In addition, we need the dimensional functions $k$ (Robin coefficient) and $\rho_I$ (initial liquid density). Except for the Robin coefficient $k$, all the model parameters and functions are either known or can be accessed directly via measurements.
Even if the boundary $\Gamma_R$ is not accessible for measurements of parameters such as $\kappa$, this can be compensated by measuring on the boundary $\Gamma_N:=\partial Y\setminus\Gamma_R$. See e.g. \cite{lind16}.
We would like to point out that although we choose $A$ and $D$ constant, the analysis would be analogous for a system with $A = A(x)$ and $D = D(y)$, provided they satisfy certain suitable assumptions.

In this context, we prove existence and uniqueness of a discrete-in-space, continuous-in-time finite element element approximation and prove its convergence to the true solution of ($P_1$).
The main results of this contribution are the well-posedness of the Galerkin approximation (Theorem~\ref{thm:ex_un}), convergence rates for the approximation (Theorem~\ref{thm:priori}), and the confirmation of the expected convergence rate with a numerical simulation (Section~\ref{sec:implementation}).

The choice of problem and approach is in line with other investigations running for two-scale systems, or systems with distributed microstructures, such as \cite{lind15,sebamPhD,PesShow}. The reader is also referred to the FE$^2$ strategies developed by the engineering community to describe the evolution of mechanical deformations in structured heterogeneous materials; see e.g. \cite{Varvara} and references cited therein. Other classes of computationally challenging two-scale problems are mentioned, for instance, in \cite{Redeker}, where the pore scale model has \textit{a priori} unknown boundaries, and in \cite{Ijioma} for a smoldering combustion scenario. This paper continues an investigation started in related works.
In \cite{lind16}, we study the solvability issue and derive inverse Robin estimates for a variant of this model problem. Two-scale Galerkin approximations have been derived previously for related problem settings; see e.g. our previous investigations \cite{radu10}, \cite{RIMS}, \cite{Chal}, and \cite{lind15}.

The rest of this paper is structured as follows.
In Section~\ref{sec:preliminaries}, we discuss the technical concepts and requirements we need before starting our analysis.
Then, in Section~\ref{sec:wellposedness}, we show the Galerkin approximation is well-posed and converges to the weak solution of the original system.
In Section~\ref{sec:convergence}, we prove \emph{a priori} convergence rates for the Galerkin approximation.
Next, in Section~\ref{sec:implementation}, we propose a fully discrete scheme, an implementation of said scheme, and approximation errors of the finite element solutions on subsequently refined grids.
Finally, in Section~\ref{sec:conclusion}, we conclude this paper and provide an outlook into future research.
\section{Concept of weak solution, assumptions and technical preliminaries}
\label{sec:preliminaries}

\subsection{Weak solutions}
We look for solutions to ($P_1$) in the weak sense. This is motivated by the fact that the underlying structured media can be of composite type, allowing for discontinuities in the model parameters. However, already at this stage it is worth mentioning that the solutions to ($P_1$) are actually more regular than stated, i.e. with minimal adaptations of the working assumptions the regularity of the solutions can be lifted so that they turn out to be strong or even classical. We will lift their regularity only when needed.

\begin{definition}[Weak solution]
    A weak solution of ($P_1$) is a pair $(\pi,\rho)\in L^2(S;\H_0^1(\Omega))\times  L^2(S; L^2(\Omega;\H^1(Y)))$ 
    such that $\partial_t\rho\in L^2(S\times\Omega\times Y)$ and for all test functions $(\varphi,\psi) \in \H^1_0(\Omega)\times  L^2(\Omega;\H^1(Y))$ the following identities are satisfied

\begin{equation}
    \label{eq:weak_pi_cont}
    A\int_\Omega\nabla_x\pi\cdot\nabla_x\varphi dx=\int_\Omega {f(\pi,g(\rho))}\varphi dx,
\end{equation}
and
\begin{equation}
    \label{eq:weak_rho_cont}
    \int_\Omega\int_Y\partial_t\rho\psi dydx+D\int_\Omega \int_Y\nabla_y \rho\cdot\nabla_y\psi dydx= \kappa\int_\Omega\int_{\Gamma_R}(\pi+p_F-R\rho)\psi d\sigma_ydx,
\end{equation}
for almost every $t \in S$. Furthermore, we require that $\rho(0,x,y)=\rho_I(x,y)$, which is provided. For $\pi$, we have $\pi(0,x)=\pi_I(x)$ where $\pi_I\in \H^1_0(\Omega)$ is the weak solution of
\begin{equation}
    \begin{split}
        -A\Delta_x\pi&=f(\pi,g(\rho_I))  \mbox{ in }\Omega,\\
        \pi &= 0  \mbox{ in }\partial\Omega.
    \end{split}
    \label{eq:init_pi}
\end{equation}
Note that \eqref{eq:init_pi} is a stationary elliptic equation giving access to the value of $\pi$ from \eqref{eq:main_ellip} at time $t=0$.
Above, $d\sigma_y$ denotes the surface measure on $\partial Y$.
\end{definition}

\subsection{Assumptions}

We introduce a set of assumptions that allows us to ensure the weak solvability and approximation of $(P_1)$.
\begin{enumerate}[label=(\subscript{A}{\arabic*})]
    \item\label{as:lips_boundary} The domains $\Omega$ and $Y$ are convex polygons.

    \item\label{as:positive_params} All model parameters are positive; in particular $A, D, R, p_F$, $\kappa$.
    
    \item\label{as:parameterCond} The parameter $A$ satisfies
    $$
    \frac{2\mathcal{C}_\Omega}{A}<1
    $$
    where $\mathcal{C}_\Omega$ denotes the constant in Poincar\'e's inequality for $\H^1_0(\Omega)$ (see (\ref{eq:poincareIneq}) below).

    \item\label{as:smooth_initial} $\rho_I \in  L^2(\Omega; \H^1(Y))$.

    \item\label{as:rhs} $f:\mathbb{R}^2\rightarrow\mathbb{R}$ in \eqref{eq:main_ellip} satisfies the following conditions:
        \begin{enumerate}[label=({\roman*})]
            \item the weak partial derivatives $D_1f, D_2f\in L^\infty(\mathbb{R}^2)$ and
            \begin{equation*}
            \|D_1f\|_{ L^\infty(\mathbb{R}^2)}\le\theta\quad{\rm and}\quad\|D_2f\|_{ L^\infty(\mathbb{R}^2)}\le\theta
            \end{equation*}
            where $\theta$ is small enough to satisfy all of the following:
            \begin{equation*}
                \theta<\frac{2\mathcal{C}_\Omega}{A},\quad\theta<\frac{1}{\mathcal{C}_\Omega},\quad 1+4\theta^2<\frac{3}{2}.
            \end{equation*}
            Note that this implies $\theta<1$.
            \item $f(0,s)=0$ for all $s\in\mathbb{R}$,
            \item $|f(r,s)|\le C_f\min(|r|,|r|^\alpha)$ for some constant $C_f>0$ and some $\alpha\in(0,1)$ and all $s\in\mathbb{R}$,
        \end{enumerate}
        \item\label{as:g-functional}
        $g$ is a linear functional such that
        \begin{equation*}
        \int_\Omega g(\rho)^2dx\le C_g\|\rho\|^2_{ L^2(\Omega;\H^1(Y))}
        \end{equation*}
        for some constant $C_g>0$.
\end{enumerate}
\begin{remark}
\ref{as:lips_boundary}, \ref{as:positive_params} and \ref{as:smooth_initial} are straightforward assumptions related to the physical setting.
In \ref{as:positive_params}, all parameters are constant as a result of the periodic homogenization procedure behind the structure of our model. A deviation from periodicity would introduce an $x$-dependence in the coefficients. See \cite{lic} for a derivation of the model under consideration.
\ref{as:lips_boundary} is a condition to ease the interaction with the finite element mesh.
\ref{as:rhs} and \ref{as:g-functional} are technical conditions required to prove well-posedness of the problem.
\end{remark}
\begin{remark}
Note that the condition $\|D_jf\|_{ L^\infty(\mathbb{R}^2)}\le\theta<1$ implies that $f(\cdot,s)$ and $f(r,\cdot)$ are contractions for any $(r,s)\in\mathbb{R}^2$.
\end{remark}
\begin{remark}
Examples of a nonlinearity $f$ satisfying \ref{as:rhs} is
\begin{equation*}
    f(r,s)=\theta\min(|r|,|r|^\alpha)\min(1,|s|),
\end{equation*}
for some $\alpha \in (0,1)$ or 
\begin{equation*}
    f(r,s) = |\theta\sin( r)\cos(s)|.
\end{equation*}
An example of a functional $g$ satisfying \ref{as:g-functional} is
\begin{equation}
\nonumber
 g(\rho)(t,x)=\int_{\Gamma_R}L(\rho(t,x,y))d\sigma_y, 
\end{equation}
where $L$ is a linear map. The fact that $g$ defined as above satisfies \ref{as:g-functional} is a consequence of the interpolation-trace inequality (see (\ref{eq:trace_ineq}) below).
\end{remark}

\subsection{Technical preliminaries}
\label{sec:notation}

The rest of this section introduces the notation of the functional spaces and norms used in the paper.
Let $\operatorname{Tr_1}: L^2(\Omega)\rightarrow L^2(\partial\Omega)$ denote the (macroscopic) trace operator defined as
\begin{equation*}
    \operatorname{Tr_1}(u)=\left.u\right|_{\partial\Omega},
\end{equation*}
and let $\operatorname{Tr_2}: L^2(\Omega;\H^1(Y))\rightarrow L^2(\Omega\times\partial Y)$ denote the microscopic trace operator defined as
\begin{equation*}
    \operatorname{Tr_2}(u)=\left.u\right|_{\Omega\times\partial Y}.
\end{equation*}

Let $f,g: D \to \R$. Then the Lebesgue and Sobolev norms are defined as follows:
\begin{equation}
 \left|\left|f\right|\right|_{ L^p(D)} := \begin{cases}
     \left( \int_D|f(x)|^pdx\right)^{1/p} &\mbox{ for }1\leq p < \infty,\\
     \operatorname{ess}\operatorname{sup} \left\{ |f(x)| : x \in D \right\}&\mbox{ for }p = \infty,
    \end{cases}
\end{equation}

\begin{equation}
 \left|\left|f\right|\right|_{\H^k(D)} :=
 \left(\sum_{|\alpha|\leq k} \int_D \left|\partial^\alpha f\right|^2 dx\right)^{1/2}\mbox{ for }  k \in \N,
\end{equation}
with $\partial^\alpha f$ denoting derivatives in the weak sense.

Furthermore, for $ L^2(D)$ and $\H^k(D)$ we have the following inner products.
\begin{equation}
    \langle f,g \rangle_{ L^2(D)} := \int_D f(x)g(x) dx,
\end{equation}

\begin{equation}
    \langle f,g \rangle_{\H^k} := \sum_{|\alpha|\leq k} \langle \partial^\alpha f,\partial^\alpha g\rangle_{ L^2(D)}.
\end{equation}

Moreover, we use $\H^1_0(D)$ to denote the following function space:
\begin{equation}
    \H^1_0(D) := \left\{ u \in \H^1(D): \operatorname{Tr}_1(u)= 0\right\}.
    \label{eq:sobolev_compact}
\end{equation}

Let $B$ be a Banach space with norm $||\cdot||_B$. Then $u$ belongs to Bochner space $ L^2(S;B)$ if it has a finite $ L^2(S;B)$ norm, defined as follows:
\begin{equation}
    \left|\left|u\right|\right|_{ L^2(S,B)} := \left( \int_S||u(t)||_B^2dt\right)^{1/2}.
\end{equation}

An introduction to the concepts of Lebesgue and Bochner integration as well as on inner products and norms can be found in many functional analysis textbooks (e.g. \cite{adams2003}).

\subsection{Auxiliary results}
For the benefit of the reader, we collect a number of well-known results that we will use in the paper.

\begin{lemma}[Young's inequality]
\label{lem:young_ineq}
Let $E\subseteq\mathbb{R}^d$ be a measurable set and $u,v\in L^2(E)$. For any $\epsilon>0$ there holds
\begin{equation}
\int_E|u(x)v(x)|dx\le \epsilon\|u\|^2_{ L^2(E)}+\frac{1}{4\epsilon}\|v\|^2_{ L^2(E)}.
\end{equation}
\end{lemma}

It is well-known that $\operatorname{Tr_2}$ defined between the function spaces specified above is a bounded linear operator. Thus, quantities of the type $\|u\|_{ L^2(\Omega\times\partial Y)}$ are well defined for $u\in  L^2(\Omega;\H^1(Y))$.
\begin{lemma}[Interpolation-trace inequality \cite{ladyzenskaja}]
\label{lem:trace_ineq}
 Let $u\in  L^2(\Omega;\H^1(Y))$, then for any $\epsilon>0$ there holds
\begin{equation}
\|u\|^2_{ L^2(\Omega\times\Gamma_R)}\le\varepsilon \|\nabla_y u\|^2_{ L^2(\Omega\times Y)} +C^*\max(\epsilon,\epsilon^{-1})\|u\|^2_{ L^2(\Omega\times Y)}
        \label{eq:trace_ineq}
    \end{equation}
where $C^*$ is a constant depending only on $Y$ and $\Gamma_R$.
\end{lemma}

\begin{lemma}[Poincar\'e's inequality]
There exists a constant $\mathcal{K}_\Omega$ depending only on $\Omega$ such that 
\begin{equation}
\label{eq:poincareIneq}
\|u\|^2_{ L^2(\Omega)}\le\mathcal{C}_\Omega\|\nabla_xu\|^2_{ L^2(\Omega)}    
\end{equation}
for all $u\in\H^1_0(\Omega)$.
\end{lemma}

\begin{lemma}[Aubin-Lions lemma]
    \label{lem:aubinlions}
Let $B_0 \subset\subset B \embed B_1$ be Banach spaces, i.e. $B_0$ be compactly embedded in $B$ and $B$ be continuously embedded in $B_1$. Let
\begin{equation}
    W := \left\{ u\in  L^2\left( S;B_0 \right) | \partial_t u \in  L^2\left( S; B_1 \right) \right\}.
    \label{eq:aubinlions}
\end{equation}
Then the embedding of $W$ into $ L^2\left( S;B \right)$ is compact.
\end{lemma}
We refer the reader to \cite{aubin1963} for the original proof of the statement.
 \section{Well-posedness}
\label{sec:wellposedness}
In this section we prove that ($P_1$) has a weak solution by approximating it with a Galerkin projection.
We show the projection exists and is unique, and proceed by proving it converges to the weak solution of ($P_1$).
First, we introduce the necessary tools for defining the Galerkin approximation.

We use one mesh partition for each of the two spatial scales.
Let $\mathbb{P}^k$ be the space of polynomials up to degree $k$.
Let $\B_H$ be a mesh partition for $\Omega$ consisting of simplices. We denote the diameter of an element $B \in \B_H$ with $H_B$, and the global mesh size with $H:= \max_{B \in \B_H} H_B$.
We introduce a similar mesh partition $\K_h$ for $Y$ with global mesh size $h:= \max_{K \in \K_h} h_K$.

Our macroscopic and microscopic finite element spaces $V_H$ and $W_h$ are, respectively:
\begin{align*}
    V_H &:= \left\{ \left. v \in \mathbf{C}(\bar{\Omega})\right|\,v|_B \in \mathbb{P}^1(B) \mbox{ for all } B \in \B_H,\, v=0 \mbox{ on } \partial \Omega  \right\},\\
        W_h &:= \left\{ \left. w \in \mathbf{C}(\bar{Y})\right|\,w|_K \in \mathbb{P}^1(K) \mbox{ for all } K \in \K_h  \right\}.
\end{align*}
We note that this approach is easily extensible to $x$-dependent microscopic domains.
Let $\mathcal{N}_1$ and $\mathcal{N}_2$ denote the sets of degrees of freedom in $\B_H$ and $\K_h$, respectively.
Let $ \operatorname{span}\left( \xi_i \right) = V_H$ and $ \operatorname{span}\left( \eta_k \right) = W_h$, and let $\alpha_i,\beta_{ik} : S \to \R$ denote the Galerkin projection coefficient for the $i$th and $ik$th degree of freedom, respectively. We introduce the following finite-dimensional Galerkin approximations of the functions $\pi$ and $\rho$:
\begin{equation}
    \begin{split}
        \pi^H(t,x) &:= \sum_{i\in \mathcal{N}_1} \alpha_i(t) \xi_i(x),\\
        \rho^{H,h}(t,x,y) &:= \sum_{i\in \mathcal{N}_1, k\in \mathcal{N}_2} \beta_{ik}(t) \xi_i(x) \eta_k(y).
    \end{split}
    \label{eq:trunc}
\end{equation}

Reducing the space of test functions to $V^H$ and $W^h$, we obtain the following discrete weak formulation: find a solution pair $(\pi^H,\rho^{H,h}) \in  L^2(S;V^H)\times  L^2(S;V^H\times W^h)$ where $\partial_t \rho^{H,h} \in  L^2(S;V^H\times W^h)$ that solve
\begin{equation}
    \label{eq:weak_pi}
    A\int_\Omega\nabla_x\pi^H\cdot\nabla_x\varphi dx=\int_\Omega f(\pi^H,g(\rho^{H,h}))\varphi dx,
\end{equation}
and
\begin{equation}
    \label{eq:weak_rho}
    \begin{split}
    &\int_\Omega\int_Y\partial_t\rho^{H,h}\psi dydx+D\int_\Omega \int_Y\nabla_y \rho^{H,h}\cdot\nabla_y\psi dydx\\
    &\quad= \kappa\int_\Omega\int_{\Gamma_R}(\pi^H+p_F-R\rho^{H,h})\psi d\sigma_ydx,
    \end{split}
\end{equation}
for any $\varphi \in V_H$ and $\psi \in V_H \times W_h$ and almost every $t \in S$. Furthermore, $\pi^H(0,x)=\pi_I^H(x)$ (see \eqref{eq:init_pi} and $\rho^{H,h}(0,x,y)=\rho^{H,h}_I(x,y)$ where
$(\pi_I^H,\rho_I^{H,h})\in V_H\times (V_H\times W_h)$ is a Galerkin approximation of $(\pi_I,\rho_I)$.
These concepts lead us to the first theorem.
\ \\
\begin{theorem}[Existence and uniqueness of the Galerkin approximation]
    \label{thm:ex_un}
    There exists a unique solution $(\pi^H,\rho^{H,h})$ to the system in \eqref{eq:weak_pi}-\eqref{eq:weak_rho}.
\end{theorem}
\begin{proof}
    The proof is divided in three steps. In step 1, the local existence in time is proven. In step 2, global existence in time is proven. Step 3 is concerned with proving the uniqueness of the system.

    \textit{Step 1: local existence of solutions to \eqref{eq:weak_pi} - \eqref{eq:weak_rho}}:
    By substituting $\varphi = \xi_i$ and $\psi = \xi_i\eta_k$ for $i\in \mathcal{N}_1$ and $k\in \mathcal{N}_2$ in \eqref{eq:weak_pi}-\eqref{eq:weak_rho} we obtain the following system of ordinary differential equations coupled with algebraic equations: find $(\alpha, \beta) \in \mathbf{C}(S) \times \mathbf{C}^1(S)$ such that
    \begin{align}
        &\sum_{j\in \mathcal{N}_1} P_{ij} \alpha_j(t) = F_i(\alpha(t),\beta(t))\label{eq:ode_alpha} \mbox{ for }i\in \mathcal{N}_1,\\
        &\sum_{j\in \mathcal{N}_1,l\in \mathcal{N}_2} M_{ijkl} \beta^\prime_{ik}(t) + \sum_{j\in \mathcal{N}_1,l\in \mathcal{N}_2} Q_{ijkl} \beta_{jl}(t) \\
        &= c_{ik} + \sum_{j\in \mathcal{N}_1} E_{ijk}\alpha_j(t)\label{eq:ode_beta}\mbox{ for }i\in \mathcal{N}_1\mbox{ and }k \in \mathcal{N}_2,
    \end{align}
    with
    \begin{equation}
        \begin{split}
            P_{ij} &:= A \int_\Omega \nabla_x \xi_i \cdot \nabla_x \xi_j\,dx,\\
            F_i &:= \int_\Omega f \left( \sum_{j\in \mathcal{N}_1}\alpha_j(t) \xi_j,\sum_{j\in \mathcal{N}_1,l\in \mathcal{N}_2}\beta_{jl}(t)\xi_j\eta_l \right)\xi_i\,dx,\\
            M_{ijkl} &:= \int_\Omega \xi_i\xi_jdx\int_{Y}\eta_k \eta_l dy,\\
            Q_{ijkl} &:= D \int_\Omega \xi_i\xi_jdx\int_{Y} \nabla_y \eta_k \cdot \nabla_y \eta_ldy + \kappa R\int_\Omega \xi_i\xi_j dx\int_{\Gamma_R}\eta_k \eta_ld\sigma_y,\\
            E_{ijk} &:= \kappa\int_\Omega\xi_i\xi_j dx\int_{\Gamma_R} \eta_kd\sigma_y,\\
            c_{ik}&:= \kappa p_F\int_\Omega \xi_i dx\int_{\Gamma_R} \eta_k d\sigma_y\,
        \end{split}
        \label{eq:matrices}
    \end{equation}
    Applying \eqref{eq:main_initial} to \eqref{eq:weak_pi} and \eqref{eq:weak_rho} yields:
    \begin{equation}
        \begin{split}
            \alpha_i(0) &= \int_\Omega \xi_i\pi_I\,dx,\\
            \beta_{ik}(0) &= \int_\Omega\int_Y\xi_i\eta_k\rho_Idydx.
        \end{split}
        \label{eq:fem_init}
    \end{equation}

    For all $t \in S$, the coefficients $\alpha_i(t), \beta_{ik}(t)$ of \eqref{eq:trunc} are determined by \eqref{eq:ode_alpha}, \eqref{eq:ode_beta} and \eqref{eq:fem_init}.

    Since the system of ordinary differential equations in \eqref{eq:ode_beta} is linear, we are able to explicitly formulate the solution representation for $\beta_{ik}$ with respect to $\alpha_i$. Let $\alpha_i$ be given, and let $Q$ and $E$ denote matrices given by:
    \begin{align}
        Q\beta &= \sum_{j\in \mathcal{N}_1,l\in \mathcal{N}_2} Q_{ijkl} \beta_{jl},\\
        M\beta &= \sum_{j\in \mathcal{N}_1,l\in \mathcal{N}_2} M_{ijkl} \beta_{jl},\\
        E\alpha &= \sum_{j\in \mathcal{N}_1} E_{ijk}\alpha_j.
    \end{align}
    Then $\beta_{ik}$ can be expressed as
    \begin{equation}
        \beta_{ik}(t) = M^{-1}(\beta_{ik}(0)e^{-Qt} + Q^{-1}(c+E\alpha_i(t))(I - e^{-Qt})).
        \label{eq:matrix_exponential}
    \end{equation}
    \label{lem:b_from_a}
    Substituting \eqref{eq:matrix_exponential} in \eqref{eq:ode_beta} results in the expression:
    \begin{equation}
        \begin{split}
            &(c+E\alpha_j(t))e^{-Qt} + M^{-1}\left(Q\beta_{ik}(0)e^{-Qt} + (c+E\alpha_j(t))(I-e^{-Qt})\right) \\
             &= Q\beta_{ik}(0)e^{-Qt} +c+E\alpha_j(t).
        \end{split}
    \end{equation}
    \ref{as:rhs} implies that for all $i \in \mathcal{N}_1$, $F_i$ are contractions.
    Let $(\alpha^*,\beta^*) = \left((\alpha^*_i)_i, (\beta^*_{ik})_{ik}\right)$ and $(\alpha^{**},\beta^{**}) = \left((\alpha^{**}_i)_i, (\beta^{**}_{ik})_{ik}\right)$ be two function pairs that satisfy \eqref{eq:matrix_exponential}.
    Then it holds that
    \begin{equation}
      \begin{split}
          &|F_i(\alpha^{*},\beta^{*}) - F_i(\alpha^{**},\beta^{**})|,\\
          &\leq |F_i(\alpha^{*},\beta^{*}) - F_i(\alpha^{*},\beta^{**}) + F_i(\alpha^{*},\beta^{**}) - F_i(\alpha^{**},\beta^{**})|, \\
          &= \left|\int_\Omega f \left( \sum_{j\in \mathcal{N}_1}\alpha^{*}_j(t) \xi_j,\sum_{j\in \mathcal{N}_1,l\in \mathcal{N}_2} \beta^{*}_{jl}(t) \xi_j\eta_l \right)\xi_i\right.\\
          &\quad- \left. f \left( \sum_{j\in \mathcal{N}_1}\alpha^{*}_j(t) \xi_j,\sum_{j\in \mathcal{N}_1,l\in \mathcal{N}_2}\beta^{**}_{jl}(t)\xi_j\eta_l \right)\xi_i\,dx\right|\\
          &+ \left|\int_\Omega f \left( \sum_{j\in \mathcal{N}_1}\alpha^{*}_j(t) \xi_j,\sum_{j\in \mathcal{N}_1,l\in \mathcal{N}_2}\beta^{**}_{jl}(t)\xi_j\eta_l \right)\xi_i\right.\\
          &\quad- \left. f \left( \sum_{j\in \mathcal{N}_1}\alpha^{**}_j(t) \xi_j,\sum_{j\in \mathcal{N}_1,l\in \mathcal{N}_2}\beta^{**}_{jl}(t)\xi_j\eta_l \right)\xi_i\,dx\right|,\\
          &\leq \left|\int_\Omega c_\rho \sum_{j\in \mathcal{N}_1,l\in \mathcal{N}_2}(\beta^{*}_{jl}(t) - \beta^{**}_{jl}(t))\xi_j\eta_l
          + c_\pi \sum_{j\in \mathcal{N}_1}(\alpha^{*}_j(t) - \alpha^{**}_j(t)) \xi_j \,dx\right|,\\
          &\leq c_\beta \left|\sum_{j\in \mathcal{N}_1}\beta^{*}_{jl}(t) - \beta^{**}_{jl}(t)\right| +  c_\alpha \left|\sum_{j\in \mathcal{N}_1}\alpha^{*}_j(t) - \alpha^{**}_j(t)\right|,
      \end{split}
      \label{eq:lipschitz_carry}
    \end{equation}
    with $c_\alpha, c_\beta$ defined as
    \begin{align}
        c_\alpha:= c_\pi\max_{j\in N_1}\int_\Omega \xi_j\,dx\leq c_\pi,& &c_\beta := c_\rho\max_{j\in \mathcal{N}_1,l\in \mathcal{N}_2}\int_\Omega \xi_j\eta_ldx\leq c_\rho.
    \end{align}
    Now, we derive a time-dependent continuity estimate for sufficiently small $t$. Again picking two pairs $\left(\alpha^*(t),\beta^*(t)\right)$ and $\left(\alpha^{**}(t),\beta^{**}(t)\right)$ (not necessarily the same as in \eqref{eq:lipschitz_carry}):
    \begin{equation}
        \begin{split}
            ||\beta^*(t) - \beta^{**}(t)|| &= ||I-e^{Qt}||\cdot||M^{-1}Q^{-1}E||\cdot||\alpha^*(t) - \alpha^{**}(t)||,\\
            & = ||Qt + \bigo{t^2}||\cdot||M^{-1}Q^{-1}E||\cdot||\alpha^*(t) - \alpha^{**}(t)||,\\
            &\leq tC||\alpha^* - \alpha^{**}|| \mbox{ for small $t$}.
        \end{split}
        \label{eq:time_lp}
    \end{equation}
    Here, the size of valid $t$ is independent of the initial data.
    Using \eqref{eq:time_lp} we obtain a Lipschitz bound on all $F_i$ in the interval $[0,\tau]$ for any choice of $\tau < t$:
    \begin{equation}
      \begin{split}
          &||F_i(\alpha^*(t),\beta^*(t)) - F_i(\alpha^{**}(t),\beta^{**}(t))||\\ &\leq ||F_i(\alpha^*(t),\beta^*(t)) - F_i(\alpha^*(t),\beta^{**}(t))|| + ||F_i(\alpha^*(t),\beta^{**}(t)) - F_i(\alpha^{**}(t),\beta^{**}(t))||,\\
          &\leq c_\alpha ||\alpha^*(t) - \alpha^{**}(t)|| + c_\beta || \beta^*(t) - \beta^{**}(t)||,\\
          &\leq \left(c_\alpha + c_\beta C\tau  \right) ||\alpha^*(t) - \alpha^{**}(t)||.
      \end{split}
    \end{equation}

    Choosing $\tau$ small enough to satisfy $c_\alpha + c_\beta C\tau<1$ makes $F$ a contraction on $[0,\tau]$.
    By Banach's fixed point theorem, it follows that the equation $F(\alpha(t),\beta(t)) = \alpha(t)$ has a solution for $\alpha$ in $ L^2(S)$.
    Substitution of $\alpha(t)$ into \eqref{eq:matrix_exponential} leads to the corresponding $\beta$.
    Existence of $\pi^H$ and $\rho^{H,h}$ follows directly.

    \textit{Step 2: global existence of solutions to \eqref{eq:weak_pi} - \eqref{eq:weak_rho}}:
    We cover time interval $S$ into $N$ intervals of length at most $\tau$ such that $S \subseteq \bigcup_n((n-1)\tau,n\tau]$.
    From the arguments in the previous paragraph it is clear a solution exists on the first interval $[0,\tau]$. This allows us to provide an induction argument for the existence of a solution on interval $n$:

    Given that interval $n$ has local solution $\beta \left( \left[ (n-1)\tau,n\tau \right] \right)$, we can obtain values $\beta(n\tau)$, $\beta^\prime(n\tau)$, $\alpha(n\tau)$ as initial values to the local system on interval $n+1$, and show existence of a solution. This way, we are able to construct a solution satisfying \eqref{eq:weak_pi} - \eqref{eq:weak_rho} everywhere on $S$.

    \textit{Step 3: uniqueness of solutions to \eqref{eq:weak_pi} - \eqref{eq:weak_rho}}:
    We decouple the system and use a fixed point argument to show that this system has a globally unique solution in time.

    Let $(\alpha^*, \beta^*)$ and $(\alpha^{**},\beta^{**})$ be two solutions satisfying \eqref{eq:weak_pi} - \eqref{eq:weak_rho} with the same initial data.
    Let $\bar{\beta}(t) := \beta^*(t) - \beta^{**}(t)$ and $\bar{\alpha}(t) := \alpha^*(t) - \alpha^{**}(t)$.
    By starting from \eqref{eq:ode_beta} and multiplying both equations with $\bar{\beta}(t)$, we obtain
    \begin{equation}
        \begin{split}
            \langle M\bar{\beta}(t),\bar{\beta}^\prime(t)\rangle &= \langle Q\bar{\beta}(t),\bar{\beta}(t)\rangle + \langle E\bar{\alpha}(t),\bar{\beta}(t)\rangle,\\
            \frac{1}{2}\frac{d}{dt} \norm{\bar{\beta}(t)}^2 &\leq \norm{M^{-1}Q}\norm{\bar{\beta}(t)}^2 + \norm{M^{-1}E}\norm{ \bar{\alpha}(t)}\ \norm{\bar{\beta}(t)}.
        \end{split}
    \end{equation}
    Since $\bar{\beta}(0)=0$, by applying Gr\"onwall's differential inequality, we know that $\bar\beta(t)\equiv0$. Combined with \eqref{eq:matrix_exponential}, it immediately follows that $\bar\alpha(t)\equiv0$, and therefore, $(\alpha^*, \beta^*)=(\alpha^{**},\beta^{**})$.

\end{proof}
Note that showing the stability of the finite element approximation with respect to data and initial conditions follows an analogous argument. The proof is omitted here.

The remaining part of this section is devoted to proving that the system in \eqref{eq:weak_pi}-\eqref{eq:weak_rho} converges to the solution of the Galerkin projection converges to the weak solution of ($P_1$).
Our first aim is to derive standard energy estimates for the discrete solution $(\pi^H,\rho^{H,h})$.
\begin{lemma}[Standard energy estimates]
\label{EnergyEstimates}
Let $(\pi^H,\rho^{H,h})$ be a solution to \eqref{eq:weak_pi}-\eqref{eq:weak_rho}. 
Then we have the following energy estimates
\begin{equation}
    \label{eq:energyPI}
    \|\pi^H\|_{ L^2(S;\H^1_0(\Omega))}\le C
\end{equation}
and
\begin{equation}
    \label{eq:energyRHO}
    \|\rho^{H,h}\|_{ L^2(S; L^2(\Omega; \H^1(Y)))}\le C
\end{equation}
and
\begin{equation}
    \label{eq:energyDtRHO}
    \|\partial_t\rho^{H,h}\|_{ L^2(S\times\Omega\times Y)}\le C
\end{equation}
where $C$ is independent on $h$ and $H$, while it depends 
on the model parameters and the geometry of the domains.
\end{lemma}
\begin{proof}
    Testing \eqref{eq:weak_pi} with $\varphi = \pi^H$ and \eqref{eq:weak_rho} with $\psi = \rho^{H,h}$ yields identities
	\begin{equation}
	      \label{eq:energy_pi}A \left|\left| \nabla_x \pi^H \right|\right|^2_{ L^2(\Omega)} = \int_\Omega f(\pi^H,g(\rho^{H,h}))\pi^H\,dx,
	\end{equation}
	and
	\begin{equation}
	    \begin{split}
	        &\label{eq:energy_rho}\frac{1}{2} \frac{d}{dt} \left|\left|\rho^{H,h}\right|\right|^2_{ L^2(\Omega\times Y)}
	      + D \left|\left|\nabla_y \rho^{H,h}\right|\right|^2_{ L^2(\Omega\times Y)}\\
	      &\quad = \int_\Omega \int_{\Gamma_R} \kappa(\pi^H + p_F)\rho^{H,h}\,d\sigma_y \,dx - \kappa R \left|\left| \rho^{H,h}\right|\right|^2_{ L^2(\Omega \times\Gamma_R)}.
	    \end{split}
	\end{equation}
	We consider first (\ref{eq:energy_pi}). By \ref{as:rhs}, Hölder's inequality and Poincar\'e's inequality we have
	\begin{eqnarray}
	\nonumber
	    A\|\nabla_x\pi^H\|^2_{ L^2(\Omega)}&=&\int_\Omega f(\pi^H,g(\rho^{H,h}))\pi^H dx\le \int_\Omega|f(\pi^H,g(\rho^{H,h}))||\pi^H|dx\\
	    \nonumber
	    &\le& C_f\int_\Omega|\pi^H|\min(|\pi^H|,|\pi^H|^\alpha)dx\le
	    C_f\int_\Omega|\pi^H|^{1+\alpha}dx\\
	    \nonumber
	    &\le&C_f|\Omega|^{(1-\alpha)/2}\|\pi^H\|^{1+\alpha}_{ L^2(\Omega)}\\
	    \nonumber
	    &\le& C_f|\Omega|^{(1-\alpha)/2}\mathcal{C}_\Omega^{1+\alpha}\|\nabla_x\pi^H\|^{1+\alpha}_{ L^2(\Omega)}
	\end{eqnarray}
	Consequently,
	$$
	\|\nabla_x\pi^H\|_{ L^2(\Omega)}\le\left(\frac{C_f}{A}\right)^{1/(1-\alpha)}\sqrt{|\Omega|}\mathcal{C}^{(1+\alpha)/(1-\alpha)}
	$$
	and by Poincar\'e's inequality 
	$$
	\|\pi^H\|_{ L^2(\Omega)}\le\left(\frac{C_f}{A}\right)^{1/(1-\alpha)}\sqrt{|\Omega|}\mathcal{C}^{2/(1-\alpha)}
	$$
	Integrating over $(0,T)$ proves (\ref{eq:energyPI}).

We proceed with (\ref{eq:energy_rho}). Using Cauchy-Schwarz' inequality and \eqref{eq:trace_ineq} to the right-hand side of \eqref{eq:energy_rho} 
	\begin{equation}
	    \begin{split}
	      \int_\Omega \int_{\Gamma_R} \kappa(\pi^H + p_F)\rho^{H,h}\,d\sigma_y \,dx
	      &\leq \kappa|\Gamma_R|\left(\|\pi^H\|_{ L^2(\Omega)}+p_F|\Omega|\right)\|\rho^{H,h}\|_{ L^2(\Omega\times\Gamma_R)}.\\
	      &\leq \kappa c_E|\Gamma_R|\left(\|\pi^H\|_{ L^2(\Omega)} + p_F|\Omega|\right)\|\rho^{H,h}\|_{ L^2(\Omega;\H^1(Y))}.
	  \end{split}
	    \label{eq:first_est}
	\end{equation}
	Then, we add to both sides of \eqref{eq:energy_rho} a term $D||\rho^{H,h}||^2_{ L^2(\Omega\times Y)}$ to get
	\begin{equation}
	  \begin{split}
	      &\frac{1}{2}\frac{d}{dt}\|\rho^{H,h}\|^2_{ L^2(\Omega\times Y)} + D\|\rho^{H,h}\|^2_{  L^2(\Omega;\H^1(Y))},\\
      &\leq D\|\rho^{H,h}\|^2_{ L^2(\Omega\times Y)} + \kappa c_E|\Gamma_R|\left(\|\pi^H\|_{ L^2(\Omega)} + p_F|\Omega|\right) ||\rho^{H,h}||_{ L^2(\Omega;\H^1(Y))}.
	  \end{split}
	\end{equation}
	After applying Young's inequality with the small parameter $\varepsilon>0$, we get
	\begin{equation}
	  \begin{split}
	      &\frac{1}{2}\frac{d}{dt}||\rho^{H,h}||^2_{ L^2(\Omega\times Y)} + (D-\varepsilon) ||\rho^{H,h}||^2_{ L^2(\Omega;H^1(Y))},\\
	      &\leq D||\rho^{H,h}||^2_{ L^2(\Omega\times Y)} + \kappa^2c_E^2c_\varepsilon |\Gamma_R|^2\left( ||\pi^H||^2_{ L^2(\Omega)} + p_F^2|\Omega|^2 \right).
	  \end{split}
	\end{equation}
	By applying Gr\"onwall's inequality we obtain the desired estimates:
	\begin{align}
	    ||\rho^{H,h}||^2_{ L^2(\Omega\times Y)} &\leq C_\rho e^{Dt},\\
	    ||\nabla_y\rho^{H,h}||^2_{ L^2(\Omega\times Y)} &\leq C_\rho + \varepsilon||\rho^{H,h}||_{ L^2(\Omega\times Y)},
	\end{align}
	with
	\[C_\rho = \kappa^2c_E^2c_\varepsilon |\Gamma_R|^2\left( ||\pi^H||^2_{ L^2(\Omega)} + p_F^2|\Omega|^2 \right).\]

Finally, testing (\ref{eq:weak_rho}) with $\partial_t\rho^{H,h}\in V_H\times W_h$, we obtain
\begin{equation}
\nonumber
\begin{split}
&\|\partial_t\rho^{H,h}\|^2_{ L^2(\Omega\times Y)}+\frac{D}{2}\frac{d}{dt}\|\nabla_y\rho^{H,h}\|^2_{ L^2(\Omega\times Y)}\\
&\le\kappa\int_\Omega\int_{\Gamma_R}(\pi^H+p_F-R\rho^{H,h})\partial_t\rho^{H,h}d\sigma_ydx\\
&\le\kappa|\Gamma_R|\left(|\Omega|p_F+|\Omega|^{1/2}\|\pi^H\|_{L^2(\Omega)}\right)-\frac{R\kappa}{2}\frac{d}{dt}\|\rho^{H,h}\|^2_{ L^2(\Omega\times\Gamma_R)}
\end{split}
\end{equation}
Set 
$$
\Theta(t)=\frac{D}{2}\|\nabla_y\rho^{H,h}\|^2_{ L^2(\Omega\times Y)}+\frac{R\kappa}{2}\|\rho^{H,h}\|^2_{ L^2(\Omega\times\Gamma_R)}
$$
then we have
$$
\|\partial_t\rho^{H,h}\|^2_{ L^2(\Omega\times Y)}+\Theta'(t)\le C
$$
Integrating over $[0,T]$ we
have
\begin{equation}
    \label{eq:rhoTime}
    \|\partial_t\rho^{H,h}\|^2_{ L^2(S\times\Omega\times Y)}\le CT+\Theta(0)-\Theta(T)\le CT+\Theta(0)\le C'
\end{equation}
where $C'$ is independent of $H,h$ and $\rho^{H,h}$.
\end{proof}We shall need a bound for $\partial_t\pi^H$ and also an estimate for $\partial_t\rho^{H,h}$ that is sharper than (\ref{eq:energyDtRHO}). 
\begin{lemma}
\label{timeIncreasedRegularity}
Let $(\pi^H,\rho^{H,h})$ be the solution to (\ref{eq:weak_pi})-(\ref{eq:weak_rho}) for $H,h>0$. Then
\begin{equation}
\|\partial_t \pi^H\|_{ L^2(S;\H^1_0(\Omega))}\le C,
\label{eq:piTime}
\end{equation}
\begin{equation}
\|\partial_t \rho^{H,h}\|_{ L^2(S, L^2(\Omega;\H^1(Y)))}\le C.
\label{eq:rhoTimeH1}
\end{equation}
\end{lemma}
\begin{proof}
Differentiate (\ref{eq:weak_pi})-(\ref{eq:weak_rho}) with respect to $t$. Let $t\in S$ be fixed but arbitrary, then $\partial_t\pi^H\in V_H$ and $\partial_t\rho^{H,h}\in V_H\times W_h$. Testing the differentiated equations with $\partial_t\pi^H$ and $\partial_t\rho^{H,h}$ respectively yield
\begin{eqnarray}
    \label{eq:temporal1}
    \|\nabla_x(\partial_t\pi^H)\|^2_{ L^2(\Omega)}=\int_\Omega(D_1f\partial_t\pi^H+D_2fg(\partial_t\rho^{H,h}))\partial_t\pi^H dx\\
    \label{eq:temporal2}
    \frac{1}{2}\frac{d}{dt}\|\partial_t\rho^{H,h}\|^2_{ L^2(\Omega\times Y)}+D\|\nabla_y(\partial_t\rho^{H,h})\|^2_{ L^2(\Omega\times Y)}=\\
    \nonumber
    =\kappa\int_\Omega\int_{\Gamma_R}\partial_t\pi^H\partial_t\rho^{H,h}d\sigma_ydx-\kappa R\|\partial_t\rho^{H,h}\|^2_{ L^2(\Omega\times\Gamma_R)}
\end{eqnarray}
Consider (\ref{eq:temporal1}). By \ref{as:rhs} and Poincar\'e's inequality 
\begin{eqnarray}
\nonumber
 \|\nabla_x(\partial_t\pi^H)\|^2_{ L^2(\Omega)}&\le&\int_\Omega\theta|\partial_t\pi^H|^2+\theta|\partial_t\pi^Hg(\partial_t\rho^{H,h})|dx\\
 \nonumber
 &\le&\theta\mathcal{C}_\Omega\|\nabla_x(\partial_t\pi^H)\|^2_{ L^2(\Omega)}+\theta\int_\Omega|\partial_t\pi^H||g(\partial_t\rho^{H,h})|dx
\end{eqnarray}
By \ref{as:rhs}, we have $\theta\mathcal{C}_\Omega<1$, hence
$$
(1-\theta\mathcal{C}_\Omega)\|\nabla_x(\partial_t\pi^H)\|^2_{ L^2(\Omega)}\le \theta\int_\Omega|\partial_t\pi^Hg(\partial_t\rho^{H,h})|dx
$$
By the Cauchy-Schwartz inequality and \ref{as:g-functional}, we get
\begin{eqnarray}
\nonumber
(1-\theta\mathcal{C}_\Omega)\|\nabla_x(\partial_t\pi^H)\|^2_{ L^2(\Omega)}&\le&\theta\|\partial_t\pi^H\|_{ L^2(\Omega)}\|g(\partial_t\rho^{H,h})\|_{ L^2(\Omega)}\\
\nonumber
&\le&\theta C_g\|\partial_t\pi^H\|_{ L^2(\Omega)}\|\partial_t\rho^{H,h}\|_{ L^2(\Omega,\H^1(Y))}\\
\nonumber
&\le&\theta C_g\mathcal{C}_\Omega\|\nabla_x(\partial_t\pi^H)\|_{ L^2(\Omega)}\|\partial_t\rho^{H,h}\|_{ L^2(\Omega,\H^1(Y))}
\end{eqnarray}
whence
\begin{equation}
\label{eq:temporal3}
\|\nabla_x(\partial_t\pi^H)\|_{ L^2(\Omega)}\le \frac{C_g\theta\mathcal{C}_\Omega}{1-\theta\mathcal{C}_\Omega}\|\partial_t\rho^{H,h}\|_{ L^2(\Omega,\H^1(Y))}.
\end{equation}
We proceed with (\ref{eq:temporal2}). By applying Young's inequality with parameter $\epsilon/\mathcal{C}_\Omega$, we get
\begin{align}
\nonumber
&\frac{1}{2}\frac{d}{dt}\|\partial_t\rho^{H,h}\|^2_{ L^2(\Omega\times Y)}+D\|\nabla_y(\partial_t\rho^{H,h})\|^2_{ L^2(\Omega\times Y)}\le\\
\nonumber
&\le\frac{\epsilon}{\mathcal{C}_\Omega}\|\partial_t\pi^H\|^2_{ L^2(\Omega)}+\left(\frac{4\kappa}{\epsilon}-\kappa R\right)\|\partial_t\rho^{H,h}\|^2_{ L^2(\Omega\times\Gamma_R)}\le\\
\nonumber
&\le\epsilon\|\nabla_x(\partial_t\pi^H)\|^2_{ L^2(\Omega)}+\left(\frac{4\kappa}{\epsilon}-\kappa R\right)\|\partial_t\rho^{H,h}\|^2_{ L^2(\Omega\times\Gamma_R)}
\end{align}
Adding $D\|\partial_t\rho^{H,h}\|^2_{ L^2(\Omega\times Y)}$ to both sides of the above inequality and choosing $\epsilon$ small enough to ensure
$$
\epsilon\left(\frac{\theta C_g\mathcal{C}_\Omega}{1-\theta\mathcal{C}_\Omega}\right)^2<\frac{D}{2},
$$
we obtain
$$
\frac{1}{2}\frac{d}{dt}\|\partial_t\rho^{H,h}\|^2_{ L^2(\Omega\times Y)}+\frac{D}{2}\|\partial_t\rho^{H,h}\|^2_{ L^2(\Omega;\H^1(Y))}\le C\|\partial_t\rho^{H,h}\|^2_{ L^2(\Omega\times\Gamma_R)}+D\|\partial_t\rho^{H,h}\|^2_{ L^2(\Omega\times Y)}
$$
By using the interpolation-trace inequality with a suitable $\epsilon$, we get
$$
C\|\partial_t\rho^{H,h}\|^2_{ L^2(\Omega\times\Gamma_R)}\le\frac{D}{4}\|\nabla_y(\partial_t\rho^{H,h})\|^2_{ L^2(\Omega\times Y)}+C\|\partial_t\rho^{H,h}\|^2_{ L^2(\Omega\times Y)}
$$
which yields
$$
\frac{1}{2}\frac{d}{dt}\|\partial_t\rho^{H,h}\|^2_{ L^2(\Omega\times Y)}+\frac{D}{4}\|\partial_t\rho^{H,h}\|^2_{ L^2(\Omega;\H^1(Y))}\le C\|\partial_t\rho^{H,h}\|^2_{ L^2(\Omega\times Y)}.
$$
Integrating over $S$ and using (\ref{eq:energyDtRHO}), we obtain (\ref{eq:rhoTimeH1}) and (\ref{eq:piTime}) follows from (\ref{eq:temporal3}) and Poincar\'e's inequality.
\end{proof}
\begin{proposition}
\label{weakConvergenceProp}
Let $(\pi^H,\rho^{H,h})\in L^2(S;V_H)\times L^2(S;V_H\times W_h)$ be solutions to \eqref{eq:weak_pi}-\eqref{eq:weak_rho} for each $H,h>0$. Then there exist functions 
\begin{equation}
\pi\in L^2(S;\H^1_0(\Omega))\quad{\rm and}\quad\rho\in L^2(S; L^2(\Omega;\H^1(Y)))\cap\H^1(S; L^2(\Omega\times Y))
\end{equation}
and subsequences $H_j$ and $h_j~~(j\in\mathbb{N})$ such that when $j\rightarrow\infty$ we have
\begin{enumerate}[label=(\roman*)]
    \item $\pi^{H_j}\rightarrow\pi$ weakly in $ L^2(S;\H^1_0(\Omega))$;
    \item \label{weakConvRho1}$\rho^{H_j,h_j}\rightarrow\rho$ weakly in $ L^2(S; L^2(\Omega;\H^1(Y)))$;
    \item\label{weakConvRho2} $\partial_t\rho^{H_j,h_j}\rightarrow\partial_t\rho$ weakly in $ L^2(S\times\Omega\times Y)$;
    \item\label{weakConvRho3} $\left.\rho^{H_j,h_j}\right|_{\Gamma_R}\rightarrow\left.\rho\right|_{\Gamma_R}$ weakly in $ L^2(S\times\Omega\times\Gamma_R)$.
\end{enumerate}
\end{proposition}
\begin{proof}
By (\ref{eq:energyPI}), $\{\pi^H\}$ is a bounded subset of $ L^2(S;\H^1_0(\Omega))$. Then there exist a subsequence $\{H_j\}$ and a function $\pi\in L^2(S;\H^1_0(\Omega))$ such that $\pi^{H_j}$ converges weakly to $\pi$ in $ L^2(S;\H^1_0(\Omega))$ as $j\rightarrow\infty$. Further, for any $h>0$ the set $\{\rho^{H_j,h}\}$ is a bounded subset of $ L^2(S; L^2(\Omega;\H^1(Y)))\cap\H^1(S; L^2(\Omega\times Y))$ and thus there is a subsequence of $H_j$ (also denoted $H_j$), a subsequence $h_j$ and a function 
$$
\rho\in L^2(S; L^2(\Omega;\H^1(Y)))\cap\H^1(S; L^2(\Omega\times Y))
$$ 
such that $\rho^{H_j,h_j}$ converges weakly to $\rho$ in $ L^2(S; L^2(\Omega;\H^1(Y)))\cap\H^1(S; L^2(\Omega\times Y))$ as $j\rightarrow\infty$, this proves \ref{weakConvRho1} and \ref{weakConvRho2}. Finally, \ref{weakConvRho3} follows from the fact that the trace operator
$$
{\rm Tr}: L^2(S; L^2(\Omega;\H^1(Y)))\rightarrow L^2(S\times\Omega\times\Gamma_R),\quad{\rm Tr}(u)=\left.u\right|_{\Gamma_R}
$$
is a bounded linear operator, and hence preserves weak convergence.
\end{proof}
\begin{remark}
We will show below that the pair $(\pi,\rho)$ provided by Proposition \ref{weakConvergenceProp} is a weak solution to $(P_1)$. However, the convergence statements of Proposition \ref{weakConvergenceProp} are not strong enough to allow us to pass to limit in (\ref{eq:weak_pi})-(\ref{eq:weak_rho}), due to the nonlinear term $f(\pi^H,g(\rho^{H,h}))$. The next two lemmas provide additional regularity that will help us strengthen the convergence, and also be useful in Section 4.
\end{remark}

\begin{lemma}
\label{sharperEnergyEstimates}
Let $(\pi^H,\rho^{H,h})$ be the solution to (\ref{eq:weak_pi})-(\ref{eq:weak_rho}) for $H,h>0$. Then

\begin{equation}
\label{eq:rho_x}
\|\nabla_x\rho^{H,h}\|_{ L^2(S\times\Omega\times Y)}\le C,
\end{equation}
and
\begin{equation}
\label{eq:rho_yx}
\|\nabla_y(\nabla_x\rho^{H,h})\|_{ L^2(S\times\Omega\times Y)}\le C.
\end{equation}
\end{lemma}
\begin{proof}
To prove (\ref{eq:rho_x}) and (\ref{eq:rho_yx}), we adapt an interior regularity argument from \cite{evans10}, Chapter 6. For any $\delta>0$, set $\Omega_\delta=\{x\in\Omega: {\rm dist}(x,\partial\Omega)\ge\delta\}$. Then $\Omega_\delta\subset\subset\Omega$ and there is an open set $W$ such that $\Omega_\delta\subset W\subset \Omega$ and a smooth function $\zeta:\Omega\rightarrow[0,1]$ with
	\begin{equation}
	    \begin{cases}
            \zeta(x) =1 &\mbox{ for }x \in \Omega_\delta,\\
            \zeta(x) =0 &\mbox{ for }x \in \Omega\setminus W.
	    \end{cases}
	\end{equation}
We introduce the directional finite difference
	\[D_i^\lambda\rho^{H,h}:= \frac{\rho^{H,h}(t,x+\lambda e_i,y) - \rho^{H,h}(t,x,y)}{\lambda }\mbox{ for }\lambda >0.\]
    for $i\in\{1,...,d_1\}$. Let $0<\lambda<\delta$ and test \eqref{eq:weak_rho} with
	\[\psi = -D_i^{-\lambda}\zeta^2D_i^\lambda \rho^{H,h},\] which gives us:
	\begin{equation}
	     -\int\partial_t\rho D_i^{-\lambda}\zeta^2D_i^\lambda \rho-D\int\nabla_y\rho\cdot\nabla_y D_i^{-\lambda }\zeta^2D_i^\lambda\rho
	     = -\kappa\int(\pi + p_F - R\rho)D_i^{-\lambda}\zeta^2D_i^\lambda \rho.
	    \label{eq:weak_cutoff}
	\end{equation}
Because of the properties of the support of $\zeta$, it holds that for any $f\in \Omega $
	\begin{equation}
	    \int_{\Omega}\psi D^{-\lambda}_if = -\int_{\Omega}f D^{\lambda}_i\psi.
        \label{eq:diffness}
	\end{equation}
Applying the property in \eqref{eq:diffness} to \eqref{eq:weak_cutoff} yields
	\begin{equation}
	\begin{split}
	    &\int_{\Omega\times Y}\zeta^2D_i^\lambda \partial_t\rho^{H,h} D_i^\lambda \rho^{H,h}+ D\int_{\Omega\times Y} \zeta^2D_i^{\lambda }\nabla_y\rho^{H,h}\cdot D_i^\lambda\nabla_y \rho^{H,h}\\
	    &\quad = \kappa\int_{\Omega\times \Gamma_R}\zeta^2D_i^\lambda(\pi^H + p_F - R\rho^{H,h})D_i^\lambda\rho^{H,h},
	\end{split}
	\end{equation}
	leading to
	\begin{equation}
    	\begin{split}
    	    &\frac{1}{2}\frac{d}{dt}\int_{\Omega\times Y} \left|\zeta D_i^\lambda\rho^{H,h}\right|^2+ D\int_{\Omega\times Y} \left|\zeta D_i^{\lambda}\nabla_y\rho^{H,h}\right|^2\\
    	    &\quad = \kappa\int_{\Omega\times \Gamma_R}\zeta^2D_i^\lambda\pi^H D_i^\lambda\rho^{H,h}- \kappa R\int_{\Omega\times \Gamma_R}\left|\zeta D_i^\lambda\rho^{H,h}\right|^2.
    	\end{split}
	    \label{eq:compact_cutoff}
	\end{equation}
	Using Young's inequality combined with the inequality, we estimate the third term of \eqref{eq:compact_cutoff} as follows:
	\begin{equation}
	    \begin{split}
		 &\kappa\int_{\Omega\times \Gamma_R}\zeta^2D_i^\lambda\pi^H D_i^\lambda\rho^{H,h}\\
		&\leq \kappa|\Gamma_R|\ ||D_i^\lambda\pi^H||_{ L^2(\Omega)}\ ||\zeta D_i^\lambda\rho^{H,h}||_{ L^2(\Omega\times \Gamma_R)},\\
		&\leq c_\varepsilon\kappa|\Gamma_R|\ ||D_i^\lambda\pi^H||^2_{ L^2(\Omega)}+ \varepsilon||\zeta D_i^\lambda\rho^{H,h}||^2_{ L^2(\Omega\times \Gamma_R)},\\
		&\leq C_\varepsilon\kappa|\Gamma_R|\ ||\nabla_x \pi^H||^2_{ L^2(\Omega)} + \varepsilon||\zeta D_i^\lambda\rho^{H,h}||_{\Omega\times Y}||\zeta D_i^\lambda\nabla_y\rho^{H,h}||_{ L^2(\Omega\times Y)},\\
		&\leq C_\varepsilon\kappa|\Gamma_R|\ ||\nabla_x \pi^H||^2_{ L^2(\Omega)} + \frac{\varepsilon^2}{2}||\zeta D_i^\lambda\rho^{H,h}||_{\Omega\times Y}\\
		&\quad + \frac{\varepsilon^2}{2}||\zeta D_i^\lambda\nabla_y\rho^{H,h}||_{ L^2(\Omega\times Y)}.
	    \end{split}
	    \label{eq:cutoff_ineq}
	\end{equation}
Now, combining \eqref{eq:cutoff_ineq} with \eqref{eq:compact_cutoff}, we obtain
	\begin{equation}
    	\begin{split}
    	    &\frac{1}{2}\frac{d}{dt}\int_{\Omega\times Y} \left|\zeta D_i^\lambda\rho^{H,h}\right|^2+ \left(D-\frac{\varepsilon^2}{2}\right)\int_{\Omega\times Y} \left|\zeta D_i^{\lambda}\nabla_y\rho^{H,h}\right|^2 \\
    	    &\quad \leq C_\varepsilon\kappa|\Gamma_R|\ ||\nabla_x \pi^H||^2_\Omega + \frac{\varepsilon^2}{2}||\zeta D_i^\lambda\rho^{H,h}||_{\Omega\times Y}.
    	\end{split}
	\end{equation}
	Using Gr\"onwall's inequality, we conclude that $D_i^\lambda\rho^{H,h} \in L^2(\Omega \times Y)$, and by letting $\lambda\to0$, we obtain
\begin{equation}
\nonumber
\nabla_x\rho^{H,h}\in L^2(S\times\Omega\times Y)\quad{\rm and}\quad \nabla_y(\nabla_x\rho^{H,h})\in  L^2(S\times\Omega\times Y).
\end{equation}
\end{proof}
With the results obtained above, we are ready to state and prove the first two main result of this paper.

\begin{theorem}
\label{Theorem:conv}
The problem ($P_1$) has a weak solution, i.e. there exist functions 
$$
\pi\in L^2(S;\H^1_0(\Omega))\quad{\rm and}\quad\rho\in L^2(S; L^2(\Omega;\H^1(Y)))\cap\H^1(S; L^2(\Omega\times Y))
$$
satisfying (\ref{eq:weak_pi_cont})-(\ref{eq:weak_rho_cont}).
\end{theorem}

\begin{proof}
Let $(\pi^{H_j},\rho^{H_j,h_j})$ and $(\pi,\rho)$ be provided by Proposition \ref{weakConvergenceProp}. Inserting $(\pi^{H_j},\rho^{H_j,h_j})$ into (\ref{eq:weak_pi})-(\ref{eq:weak_rho}) and using the convergence statements of Proposition \ref{weakConvergenceProp}, it follows that $(\pi,\rho)$ solves (\ref{eq:weak_pi_cont})-(\ref{eq:weak_rho_cont}) if we can show that
\begin{equation}
\label{limit-f}
 \lim_{j\rightarrow\infty}\int_\Omega f(\pi^{H_j},g(\rho^{H_j,h_j}))\varphi dx=\int_\Omega f(\pi,g(\rho))\varphi dx   
\end{equation}
for all $\varphi\in\H^1_0(\Omega)$.
Denote by
$$
\mathcal{V}=\{u\in L^2(S;\H^1_0(\Omega)):\partial_t u\in  L^2(S\times\Omega)\}.
$$
and 
$$
\mathcal{W}=\{u\in L^2(S;\H^1(\Omega\times Y)): \partial_tu\in  L^2(S\times\Omega\times Y)\}.
$$
By the Rellich-Kondrachov theorem, we have the compact embeddings
$$
\H_0^1(\Omega)\subset\subset L^2(\Omega)
$$
and
$$
\H^1(\Omega\times Y)\subset\subset L^2(\Omega\times Y).
$$
Hence, Aubin-Lions' lemma (Lemma~\ref{lem:aubinlions}) gives that
$$
\mathcal{V}\subset\subset  L^2(S\times\Omega)\quad{\rm and}\quad\mathcal{W}\subset\subset  L^2(S\times\Omega\times Y).
$$
Note that it follows from Lemma \ref{EnergyEstimates} and Lemma \ref {sharperEnergyEstimates} that $\{\pi^{H_j}\}$ and $\{\rho^{H_j,h_j}\}$ are bounded subsets of $\mathcal{V}$ and $\mathcal{W}$ respectively. Then, by compactness, $\{\pi^{H_j}\}$ and $\{\rho^{H_j,h_j}\}$ have subsequences (also denoted $\{\pi^{H_j}\}$ and $\{\rho^{H_j,h_j}\}$) that converge strongly in the spaces $ L^2(S\times\Omega)$ and $ L^2(S\times\Omega\times Y)$ respectively, and these strong limits must coincide with $\pi$ and $\rho$. By continuity of $f$ and $g$, (\ref{limit-f}) follows, and the theorem is proved.
\end{proof}

\begin{theorem}[Uniqueness of the weak solution]
\label{uniquenessTheorem}
The weak solution to problem $(P_1)$ is unique.
\end{theorem}
\begin{proof}
Assume that there are two weak solution $(\pi_1,\rho_1)$ and $(\pi_2,\rho_2)$. Subtract the weak formulation (\ref{eq:weak_pi_cont})-(\ref{eq:weak_rho_cont}) for $(\pi_2,\rho_2)$ from that of $(\pi_1,\rho_1)$ and test with $\pi_1-\pi_2$ and $\rho_1-\rho_2$, we obtain
\begin{align}
    \label{unique:1}
    &A\|\nabla_x(\pi_1-\pi_2)\|^2_{ L^2(\Omega)}=\int_\Omega \left[f(\pi_1,\rho_1)-f(\pi_2,\rho_2)\right](\pi_1-\pi_2)dx\\
    \label{unique:2}
    \begin{split}
    &\frac{1}{2}\frac{d}{dt}\|\rho_1-\rho_2\|^2_{ L^2(\Omega\times Y)}+D\|\nabla_y(\rho_1-\rho_2)\|^2_{ L^2(\Omega\times Y)}=\\
    &=\kappa\int_\Omega\int_{\Gamma_R}\left[\pi_1-\pi_2+(p_F-R)(\rho_1-\rho_2)\right](\rho_1-\rho_2)d\sigma_ydx\\
    &=\kappa\int_\Omega\int_{\Gamma_R}(\pi_1-\pi_2)(\rho_1-\rho_2)d\sigma_y dx+\kappa(p_F-R)\|\rho_1-\rho_2\|^2_{ L^2(\Omega\times\Gamma_R)}
    \end{split}
\end{align}
We estimate the right-hand side of (\ref{unique:1}). By Young's inequality, for any $\epsilon>0$ we have
$$
A\|\nabla_x(\pi_1-\pi_2)\|^2_{ L^2(\Omega)}\le\epsilon\|\pi_1-\pi_2\|^2_{ L^2(\Omega)}+\frac{1}{4\epsilon}\int_\Omega\left[f(\pi_1,g(\rho_1))-f(\pi_2,g(\rho_2))\right]^2dx
$$
Using \ref{as:rhs} and \ref{as:g-functional}, we can estimate the second term at the right-hand side of the previous inequality as follows
\begin{equation}
    \nonumber
    \begin{split}
    &\int_\Omega\left[f(\pi_1,g(\rho_1))-f(\pi_2,g(\rho_2))\right]^2dx\le\\
    &\le2\left(\int_\Omega(f(\pi_1,g(\rho_1))-f(\pi_2,g(\rho_1)))^2dx+\int_\Omega(f(\pi_2,g(\rho_1))-f(\pi_2,g(\rho_2))^2dx\right)\\
    &\le2\theta^2\left(\|\pi_1-\pi_2\|^2_{ L^2(\Omega)}+\|g(\rho_1)-g(\rho_2)\|^2_{ L^2(\Omega)}\right)\\
    &=2\theta^2\left(\|\pi_1-\pi_2\|^2_{ L^2(\Omega)}+\|g(\rho_1-\rho_2)\|^2_{ L^2(\Omega)}\right)\\
    &\le2\theta^2\|\pi_1-\pi_2\|^2_{ L^2(\Omega)}+2C_g\theta^2\|\rho_1-\rho_2\|^2_{ L^2(\Omega;\H^1(Y))}
    \end{split}
\end{equation}
Hence,
\begin{equation}
    \nonumber
    A\|\nabla_x(\pi_1-\pi_2)\|^2_{ L^2(\Omega)}\le\left(\epsilon+\frac{\theta^2}{2\epsilon}\right)\|\pi_1-\pi_2\|^2_{ L^2(\Omega)}+\frac{C_g\theta^2}{4\epsilon}\|\rho_1-\rho_2\|^2_{ L^2(\Omega;\H^1(Y))}
\end{equation}
and by Poincar\'e's inequality we obtain
\begin{equation}
    \nonumber
    \left(\frac{A}{\mathcal{C}_\Omega}-\epsilon-\frac{\theta^2}{2\epsilon}\right)\|\pi_1-\pi_2\|^2_{ L^2(\Omega)}\le \frac{C_g\theta^2}{4\epsilon}\|\rho_1-\rho_2\|^2_{ L^2(\Omega;\H^1(Y))}.
\end{equation}
Taking $\epsilon=\theta$ we obtain
$$
\left(\frac{A}{\mathcal{C}_\Omega}-\frac{3\theta}{2}\right)\|\pi_1-\pi_2\|^2_{ L^2(\Omega)}\le\frac{C_g\theta}{4}\|\rho_1-\rho_2\|^2_{ L^2(\Omega;\H^1(Y))}.
$$
Since $\theta<2\mathcal{C}_\Omega/A<1$ it also holds that $\theta<A/(2\mathcal{C}_\Omega)$ and thus
$$
\frac{A}{\mathcal{C}_\Omega}-\frac{3\theta}{2}>2\theta-\frac{3\theta}{2}=\frac{\theta}{2}.
$$
Whence,
\begin{equation}
    \label{unique:3}
    \|\pi_1-\pi_2\|^2_{ L^2(\Omega)}\le \frac{C_g}{2}\|\rho_1-\rho_2\|^2_{ L^2(\Omega;\H^1(Y))}.
\end{equation}
By using Young's inequality with parameter $\epsilon$ (to be determined below) and (\ref{unique:3}) at the right-hand side of (\ref{unique:2}), we obtain
\begin{equation}
    \nonumber
    \begin{split}
    &\frac{1}{2}\frac{d}{dt}\|\rho_1-\rho_2\|^2_{ L^2(\Omega\times Y)}+D\|\nabla_y(\rho_1-\rho_2)\|^2_{ L^2(\Omega\times Y)}\le\\
    &\le\epsilon\kappa|\Gamma_R|\|\pi_1-\pi_2\|^2_{ L^2(\Omega)}+\left(\frac{\kappa}{4\epsilon}+\kappa|p_F-R|\right)\|\rho_1-\rho_2\|^2_{ L^2(\Omega\times\Gamma_R)}\\
    &\le\frac{C_g\epsilon\kappa|\Gamma_R|}{2}\|\rho_2-\rho_1\|^2_{ L^2(\Omega;\H^1(Y))}+\left(\frac{\kappa}{4\epsilon}+\kappa|p_F-R|\right)\|\rho_1-\rho_2\|^2_{ L^2(\Omega\times\Gamma_R)}
    \end{split}
\end{equation}
Applying the interpolation-trace inequality with parameter $\epsilon^2$, we get
\begin{equation}
    \nonumber
    \begin{split}
    &\left(\frac{\kappa}{4\epsilon}+\kappa|p_F-R|\right)\|\rho_1-\rho_2\|^2_{ L^2(\Omega\times\Gamma_R)}\\
    &\le\epsilon\left(\frac{\kappa}{4}+\mathcal{O}(\epsilon)\right)\|\nabla_y(\rho_1-\rho_2)\|^2_{ L^2(\Omega\times Y)}+\mathcal{O}(\epsilon^{-3})\|\rho_1-\rho_2\|^2_{ L^2(\Omega\times Y)}
    \end{split}
\end{equation}
Hence,
\begin{equation}
    \nonumber
    \begin{split}
    &\frac{1}{2}\frac{d}{dt}\|\rho_1-\rho_2\|^2_{ L^2(\Omega\times Y)}+D\|\nabla_y(\rho_1-\rho_2)\|^2_{ L^2(\Omega\times Y)}\le\\
    &\le\epsilon\left(\frac{\kappa(1+2C_g)}{4}+\mathcal{O}(\epsilon)\right)\|\nabla_y(\rho_1-\rho_2)\|^2_{ L^2(\Omega\times Y)}+\mathcal{O}(\epsilon^{-3})\|\rho_1-\rho_2\|^2_{ L^2(\Omega\times Y)}
    \end{split}
\end{equation} 
Now take $\epsilon_0$ such that 
\begin{equation*}
    \epsilon_0\left(\frac{\kappa(1+2C_g)}{4}+\mathcal{O}(\epsilon_0)\right)<\frac{D}{2},    
\end{equation*}
then we get
\begin{equation}
\nonumber
\frac{1}{2}\frac{d}{dt}\|\rho_1-\rho_2\|^2_{ L^2(\Omega\times Y)}+\frac{D}{2}\|\nabla_y(\rho_1-\rho_2)\|^2_{ L^2(\Omega\times Y)}\le\frac{C}{\epsilon_0^3}\|\rho_1-\rho_2\|^2_{ L^2(\Omega\times Y)},
\end{equation}
by Gr\"{o}nwall's inequality we obtain
$$
\|(\rho_1-\rho_2)(t)\|^2_{ L^2(\Omega\times Y)}\le C\|(\rho_1-\rho_2)(0)\|^2_{ L^2(\Omega\times Y)}=0
$$
in other words $\rho_1=\rho_2$ a.e., and it follows from (\ref{unique:3}) that $\pi_1=\pi_2$ a.e.
\end{proof}
 \section{Convergence rates for semidiscrete Galerkin approximations}
\label{sec:convergence}
In this section, we obtain convergence rates of the numerical approximations \eqref{eq:weak_pi} -- \eqref{eq:weak_rho}.
The following argument is largely based on standard arguments from e.g. \cite{larsson2008partial}, adapted to multiscale systems.

\begin{proposition}[Regularity lift]
    Recall \ref{as:smooth_initial} and \ref{as:lips_boundary}. Let $(\pi,\rho)$ be the weak solution to $(P_1)$. Then
    \begin{align*}
        \pi\in  L^2(S;\H^2(\Omega)),\\
        \rho\in L^2(S;\H^2(\Omega\times Y)).
     \end{align*}
\end{proposition}
\begin{proof}
We omit the proof. We refer the interested reader to \cite{evans10,grisvard11}, where the regularity lifting arguments can be adapted to fit this specific case. 
\end{proof}

Let $\ri_H : \H^1(\Omega) \to V_H$ and $\ri_h:  L^2(\Omega;\H^1(Y)) \to  L^2(\Omega;W_h)$ be the microscopic and macroscopic Ritz projection operator respectively. The projections $\ri_H r(x)$ and $\ri_h s(x,y)$ are defined such that:
\begin{align*}
    A\nabla_x (r(x) - \ri_H r(x)) \cdot \nabla_x v &= 0,\\
    D\nabla_y (s(x,y) - \ri_h s(x,y)) \cdot \nabla_y w &= 0,
\end{align*}
for all $v \in V_H$ and $w \in  L^2(\Omega;\H^1(Y))$.
\begin{lemma}[Projection error estimates]
    \label{lem:proj_error_est}
    Then there exists strictly positive constants $\gamma_l$ (with $l\in \{1,2,3,4\})$, independent of $h$ and $H$, such that projections $\ri_h\pi$ and $\ri_H\rho$ that satisfy
    \begin{align}
        \label{eq:proj_1}
        ||\pi - \ri_H\pi||_{ L^2(\Omega)} &\leq \gamma_1 H^2 ||\pi||_{\H^2(\Omega)},\\
        \label{eq:proj_2}
        ||\pi - \ri_H\pi||_{\H_0^1(\Omega)} &\leq \gamma_2 H ||\pi||_{\H^2(\Omega)},\\
        \label{eq:proj_3}
        ||\rho - \ri_H\ri_h\rho||_{ L^2(\Omega; L^2(Y))} &\leq \gamma_3(H^2+h^2) ||\rho||_{ L^2(\Omega;\H^2(Y))\cap L^2(Y;\H^2(\Omega))},
    \end{align}
    hold for all $(\pi,\rho) \in \H^2(\Omega)\times \left[  L^2(\Omega;\H^2(Y))\cap L^2(Y;\H^2(\Omega)) \right]$.
\end{lemma}
\begin{proof}
    \eqref{eq:proj_1} and \eqref{eq:proj_2} are standard Ritz projection error estimates.
    For details on the proof, see for instance \cite{thomee1984} and \cite{larsson2008partial}.
    Specific to this context, \eqref{eq:proj_3} is a two-scale estimate which accounts for the presence of the microscopic Robin boundary condition \eqref{eq:main_robin} and therefore requires some tuning. See e.g. \cite{RIMS} for similar estimates.
    Here, we only present the proof of \eqref{eq:proj_3}.

    Let $\omega:=\ri_h \rho - \rho$. Let $\varphi \in  L^2(\Omega;\H^2(Y))$ be the weak solution to
    \begin{equation}
        (P_2)\begin{cases}
        -\Delta \varphi = \omega& \mbox{ in }\Omega\times Y,\\
        -\nabla \varphi\cdot n = \alpha\varphi &\mbox{ on } \Omega\times\Gamma_R,\\
        -\nabla \varphi\cdot n = 0 &\mbox{ on } \Omega\times\Gamma_N.
    \end{cases}
    \end{equation}
    We denote the Ritz projection error of $\varphi$ with $e_\varphi$.
    By testing with $\psi$ and integrating over $\Omega\times Y$, we obtain
    \begin{equation}
        \langle \omega,\psi \rangle_{ L^2(\Omega;L^2(Y))} = \langle \nabla \varphi,\nabla \psi \rangle_{ L^2(\Omega;L^2(Y))} + \langle \nabla \varphi \cdot n,\psi \rangle_{ L^2(\Omega;L^2(\Gamma_R))}.
    \end{equation}
    Testing with $\psi=\omega$ specifically, subtracting the Galerkin approximation from the weak solution and using $(\ri_h\Delta\varphi,\omega)=0$ we obtain:
    \begin{align*}
        &||\omega||_{ L^2(\Omega;L^2(Y))}^2 \\
        &= \langle \nabla \varphi, \nabla \omega \rangle_{ L^2(\Omega;L^2(Y))} + \langle \alpha\varphi,\omega \rangle_{ L^2(\Omega;L^2(\Gamma_R))},\\
        &= \langle \nabla e_\varphi, \nabla \omega \rangle_{ L^2(\Omega;L^2(Y))} + \langle \alpha e_\varphi, \omega \rangle_{ L^2(\Omega;L^2(\Gamma_R))},\\
        &\leq c_\varepsilon \left||\nabla e_\varphi|\right|_{ L^2(\Omega;L^2(Y))} \left||\nabla \omega|\right|_{ L^2(\Omega;L^2(Y))} + \varepsilon\left||e_\varphi|\right|_{ L^2(\Omega;L^2(Y))} \left||\omega|\right|_{ L^2(\Omega;L^2(Y))}.
    \end{align*}
    Applying the Ritz projection estimates \eqref{eq:proj_1} and \eqref{eq:proj_2}, we obtain the following bound:
    \begin{equation*}
        ||\omega||_{ L^2(\Omega;L^2(Y))}^2 \leq c_\varepsilon h^2 ||\varphi||^2_{ L^2(\Omega;\H^2(Y))} + \varepsilon c h^2 ||\omega||_{ L^2(\Omega;L^2(Y))}.
    \end{equation*}
    Using Friedrich's inequality $||\varphi||_{ L^2(\Omega;\H^2(Y))} \leq C||\Delta \varphi||_{ L^2(\Omega;L^2(Y))} = C||\omega||_{ L^2(\Omega;L^2(Y))}$ for some $C$ and choosing $\varepsilon < c$ we obtain
    \begin{equation}
        (1-\varepsilon)||\omega||_{ L^2(\Omega;L^2(Y))}^2 \leq Ch^2||\omega||_{ L^2(\Omega;L^2(Y))}.
        \label{eq:omega_bound}
    \end{equation}
    \eqref{eq:omega_bound} yields:
    \begin{equation}
        ||\omega||_{ L^2(\Omega;L^2(Y))} = ||\ri_{h}\rho - \rho||_{ L^2(\Omega;L^2(Y))} \leq \bar{\gamma}_3h^2.
    \end{equation}
    Finally, we can derive \eqref{eq:proj_3} as follows:
    \begin{equation}
        \begin{split}
            &||\psi - \ri_H\ri_h\psi||_{ L^2(\Omega; L^2(Y))} = ||\psi - \ri_h \psi + \ri_h \psi - \ri_H\ri_h\psi||_{ L^2(\Omega; L^2(Y))},\\
            &\leq ||\psi - \ri_h \psi||_{ L^2(\Omega; L^2(Y))} + ||\ri_h \psi - \ri_H\ri_h\psi||_{ L^2(\Omega; L^2(Y))},\\
            &\leq \bar{\gamma}_3h^2 ||\psi||_{ L^2(\Omega;\H^2(Y))} + \bar{\gamma}_4H^2||\ri_h\psi||_{ L^2(Y;\H^2(\Omega))},\\
            &\leq \gamma_3(H^2 + h^2) ||\psi||_{ L^2(\Omega;\H^2(Y))\cap L^2(Y;\H^2(\Omega))}. 
        \end{split}
    \end{equation}
\end{proof}
By applying Lemma~\ref{lem:trace_ineq} and Lemma~\ref{lem:proj_error_est}, we can finally obtain the desired convergence rates. Let us denote the errors of the Galerkin projection as
\begin{align*}
    e_\pi := \pi - \pi^H,\\
    e_\rho := \rho - \rho^{H,h}.
\end{align*}
\begin{theorem}[Convergence rates]
Let $(\pi^H,\rho^{H,h})$ be a solution to \eqref{eq:weak_pi_cont}-\eqref{eq:weak_rho_cont} for $H,h>0$.
Then there exists a constant $C$ independent of $h$ and $H$, such that
    \begin{align}
        \label{eq:conv_rate_pi}||e_\pi||_{ L^\infty(S;  L^2(\Omega))} &\leq C(H^2 + h^2),\\
        \label{eq:conv_rate_rho}||e_\rho||_{ L^\infty(S;  L^2(\Omega\times Y))} &\leq C(H^2 + h^2).
    \end{align}
    \label{thm:priori}
\end{theorem}

\begin{proof}
By testing \eqref{eq:weak_pi_cont} and \eqref{eq:weak_pi} with $\phi^H \in V_H$ and subtracting the equations, we obtain the identity
\begin{equation}
\label{eq:weakIdent}
A\int_\Omega \nabla_x(\pi-\pi^H)\cdot \nabla_x\phi^H dx = \int_\Omega \left( f(\pi,g(\rho)) - f(\pi^H,g(\rho^{H,h}))\right)\phi^H dx
\end{equation}
for any $\phi^H\in V_H$. Let $\varphi^H\in V_H$ be arbitrary. By (\ref{eq:weakIdent}) applied with $\phi^H=\varphi^H-\pi^H$, we obtain
\begin{equation}
\nonumber
\begin{split}
&\|\nabla_xe_\pi\|^2_{ L^2(\Omega)}=\int_\Omega\nabla_x(\pi-\pi^H)\cdot\nabla_x(\pi-\pi^H)dx\\
&=\int_\Omega\nabla_x(\pi-\pi^H)\cdot\nabla_x(\pi-\varphi^H)dx+\int_\Omega\nabla_x(\pi-\pi^H)\cdot\nabla_x(\varphi^H-\pi^H)dx\\
&=\int_\Omega\nabla_xe_\pi\cdot\nabla_x(\pi-\varphi^H)dx+\\
&+\frac{1}{A}\int_\Omega\left( f(\pi,g(\rho)) - f(\pi^H,g(\rho^{H,h}))\right)(\varphi^H-\pi^H)dx.
\end{split}
\end{equation}
By the inequality above, Cauchy-Schwarz' inequality and the triangle inequality, we get
\begin{equation}
\nonumber
\begin{split}
&\|\nabla_xe_\pi\|^2_{ L^2(\Omega)}\le\|\nabla_x e_\pi\|_{ L^2(\Omega)}\|\nabla_x(\pi-\varphi^H)\|_{ L^2(\Omega)}\\
&+\frac{1}{A}\int_\Omega\left| f(\pi,g(\rho)) - f(\pi^H,g(\rho^{H,h}))\right||e_\pi|dx\\
&+\frac{1}{A}\int_\Omega\left| f(\pi,g(\rho)) - f(\pi^H,g(\rho^{H,h}))\right||\pi-\varphi^H|dx.
\end{split}
\end{equation}
Applying Cauchy-Schwarz inequality to the last two terms above and denoting 
$$
J^{H,h}=\left\|f(\pi,g(\rho))-f(\pi^H,g(\rho^{H,h})\right\|_{ L^2(\Omega)},
$$
we obtain
\begin{align}
\nonumber
&\|\nabla_xe_\pi\|^2_{ L^2(\Omega)}\le\\
\nonumber
&\le\|\nabla_xe_\pi\|_{ L^2(\Omega)}\|\nabla_x(\pi-\varphi^H)\|_{ L^2(\Omega)}
+\frac{J^{H,h}}{A}\left(\|e_\pi\|_{ L^2(\Omega)}+\|\pi-\varphi^H\|_{ L^2(\Omega)}\right)\\
\nonumber
&\le\frac{\|\nabla_xe_\pi\|^2_{ L^2(\Omega)}+\|\nabla_x(\pi-\varphi^H)\|^2_{ L^2(\Omega)}}{2}+\frac{1}{2A}\left((J^{H,h})^2+\|e_\pi\|^2_{ L^2(\Omega)}\right)\\
\nonumber
&+\frac{1}{2A}\left((J^{H,h})^2+\|\pi-\varphi^H\|^2_{ L^2(\Omega)}\right).
\end{align}
Rearranging the above inequality gives
$$
\frac{1}{2}\|\nabla_xe_\pi\|^2_{ L^2(\Omega)}-\frac{1}{2A}\|e_\pi\|^2_{ L^2(\Omega)}\le C\|\pi-\varphi^H\|^2_{\H^1_0(\Omega)}+\frac{(J^{H,h})^2}{A}
$$
Using (\ref{eq:poincareIneq}) and \ref{as:parameterCond} (i.e. $2\mathcal{C}_\Omega/A<1$), we obtain
$$
\frac{1}{2}\left(1-\frac{\mathcal{C}_\Omega}{A}\right)\|\nabla_xe_\pi\|^2_{ L^2(\Omega)}\le C\|\pi-\varphi^H\|^2_{\H^1_0(\Omega)}+\frac{(J^{H,h})^2}{A}
$$
As in the proof of Theorem \ref{uniquenessTheorem}, the assumptions \ref{as:rhs} and \ref{as:g-functional} yield the estimate
$$
(J^{H,h})^2\le 2\theta^2\|e_\pi\|^2_{ L^2(\Omega)}+2C_g\theta^2\|e_\rho\|^2_{ L^2(\Omega\times Y)}
$$
 Using (\ref{eq:poincareIneq}) again, we get
\begin{align}
\nonumber
&\frac{1}{2}\left(1-\frac{\mathcal{C}_\Omega}{A}-\frac{4\theta^2\mathcal{C}_\Omega}{A}\right)\|\nabla_xe_\pi\|^2_{ L^2(\Omega)}\le\\
\nonumber
&\le C\|\pi-\varphi^H\|^2_{\H^1_0(\Omega)}+\frac{2C_g\theta^2}{A}\|e_\rho\|^2_{ L^2(\Omega\times Y)}.
\end{align}
By the assumptions, i.e. $1+4\theta^2<3/2$ and $\mathcal{C}_\Omega/A<1/2$, we get 
$$
<te\left(1-\frac{\mathcal{C}_\Omega}{A}(1+4\theta^2)\right)>1/4.
$$
Hence, 
\begin{equation}
    \label{eq:pi-rateEstimate}
    \|\nabla_xe_\pi\|_{ L^2(\Omega)}\le C\|\pi-\varphi^H\|_{\H_0^1(\Omega)}+C\|e_\rho\|_{ L^2(\Omega\times Y)}.
\end{equation}
By (\ref{eq:proj_2}), we have $\|\pi-\varphi^H\|_{\H_0^1(\Omega)}=\mathcal{O}(H)$. Using a standard duality argument as in \cite{larsson2008partial}, we get $\|e_\pi\|_{ L^2(\Omega)}\le CH\|e_\pi\|_{\H^1_0(\Omega)}$.

We proceed to demonstrate that $\|e_\rho\|_{ L^2(\Omega\times Y)}=\mathcal{O}(H^2+h^2)$. Write
\begin{equation}
        e_\rho = \rho^{H,h} - \rho = (\rho^{H,h}- \ri_H\ri_h\rho) + (\ri_H\ri_h\rho - \rho) =: \theta + \psi.
        \label{eq:theta_psi}
    \end{equation}
    We bound $\psi$ by using Lemma~\ref{lem:proj_error_est}:
    \begin{equation}
        \begin{split}
            &||\psi(t)||_{ L^2(\Omega; L^2(Y))} \leq \gamma_3(H^2+ h^2)||\rho||_{ L^2(\Omega;\H^2(Y))\cap  L^2(Y;\H^2(\Omega))},\\
            &= \gamma_3(H^2+ h^2)\left|\left|\rho_I + \int_0^t \partial_t \rho ds\right|\right|_{ L^2(\Omega;\H^2(Y))\cap  L^2(Y;\H^2(\Omega))},
        \end{split}
        \label{eq:conv_psi}
    \end{equation}
    and bound $\theta$ from \eqref{eq:theta_psi} using the formulation: for all $\varphi \in V^h$ we have that
    \begin{equation}
      \begin{split}
          &\langle \partial_t \theta,\varphi \rangle_{ L^2(\Omega;L^2(Y))} + D\langle \nabla\theta,\nabla\varphi \rangle_{ L^2(\Omega;L^2(Y))}\\
          &= -\langle\ri_h\partial_t\rho,\varphi\rangle_{ L^2(\Omega;L^2(Y))} - D \langle \nabla\rho,\nabla\varphi \rangle_{ L^2(\Omega;L^2(Y))},\\
          &= \langle \partial_t \rho - \ri_h \partial_t\rho,\varphi \rangle_{ L^2(\Omega;L^2(Y))},\\
          &= \langle \partial_t \psi ,\varphi \rangle_{ L^2(\Omega;L^2(Y))}.
      \end{split}
      \label{eq:theta_bound}
    \end{equation}
    Substituting $\varphi = \theta$ in \eqref{eq:theta_bound} yields:
    \begin{equation}
        \begin{split}
            &\frac{1}{2}\frac{d}{dt}||\theta||_{ L^2(\Omega; L^2(Y))}^2 + D||\nabla \theta||_{ L^2(\Omega; L^2(Y))}^2 \\
            &=\left( \partial_t \rho - \ri_H\ri_h \partial_t\rho,\theta \right),\\%\leq \int \left(\partial_t \rho - R^m_h \partial_t\rho\right) \theta\\
            &\leq \left|| \partial_t \rho - \ri_H\ri_h \partial_t\rho|\right|_{ L^2(\Omega; L^2(Y))} \left||\theta|\right|_{ L^2(\Omega; L^2(Y))},\\
            & \leq \gamma_3(h^2+ H^2)||\partial_t \rho||_{ L^2(\Omega;\H^2(Y))}||\theta||_{ L^2(\Omega; L^2(Y))}.
        \end{split}
        \label{eq:theta_ode}
    \end{equation}
    Dividing the left and right hand side of \eqref{eq:theta_ode} by $||\theta||$, we obtain:
    \begin{equation}
        \begin{split}
            \frac{d}{dt} ||\theta||_{ L^2(\Omega; L^2(Y))}&\leq \gamma_3(h^2+ H^2)||\partial_t\rho||_{ L^2(\Omega;\H^2(Y))},\\
            ||\theta(t)||_{ L^2(\Omega; L^2(Y))} &\leq ||\theta(0)||_{ L^2(\Omega; L^2(Y))} + \gamma_3(h^2+ H^2)\int_0^t||\partial_t\rho||_{ L^2(\Omega;\H^2(Y))}dx,\\
            &\leq ||\rho^{H,h}_I - \rho_I||_{ L^2(\Omega; L^2(Y))} + ||\rho_I - \ri_H\ri_h\rho_I||_{ L^2(\Omega; L^2(Y))} \\
            &\quad + \gamma_3(h^2+ H^2)\int_0^t||\partial_t\rho||_{ L^2(\Omega;H^2(Y))}dx,\\
            &\leq \gamma_3(h^2+ H^2)\left(c_I + C + \int_0^t||\partial_t\rho||_{ L^2(\Omega;H^2(Y))}dx \right).
        \end{split}
        \label{eq:conv_theta}
    \end{equation}
    Because of \ref{as:smooth_initial}, the Galerkin projection error of the initial condition satisfies:
    \begin{equation}
        ||\rho_I - \rho^{H,h}_I||_{ L^2(\Omega; L^2(Y))} \leq c_I(H^2 + h^2).
    \end{equation}
    Combining \eqref{eq:conv_psi} and \eqref{eq:conv_theta} proves the desired estimate in \eqref{eq:conv_rate_rho}. 
    \begin{equation}
        ||\rho^{H,h} - \rho||_{ L^2(\Omega; L^2(Y))} = ||\theta + \psi||_{ L^2(\Omega; L^2(Y))} \leq C(H^2 + h^2)||\partial_t \rho||_{ L^2(\Omega;\H^2(Y))}.
        \label{eq:conv_in_rho}
    \end{equation}
    Finally, \eqref{eq:conv_rate_pi} follows by combining \eqref{eq:pi-rateEstimate} with \eqref{eq:conv_in_rho}.
\end{proof}

 \section{Implementation}
\label{sec:implementation}
In this section, we discuss a time-discretized version of \eqref{eq:weak_pi}-\eqref{eq:weak_rho} i.e., a fully discretized system, and provide details and performance results of the implementation of this system.

\subsection{Setup}
We implement the finite element formulation of this problem using \dealii{} (\cite{dealII}), a C++ library that computes numerical approximations to finite element problems on quadrilateral meshes.

To account for the scale separated multiscale structure, we implement this system using a heterogeneous multiscale method (see e.g. \cite{engquist07} for an introduction). Alternative multiscale finite element structures are the FEM$^2$ method (\cite{Varvara}) and the multiscale finite element method (MsFEM \cite{efendiev09}). Both of these frameworks generally require a formulation in which the size relation $\varepsilon$ between the macroscopic scale and the microscopic scale must be resolved. Since our problem is completely scale-separated, we opt for a heterogeneous multiscale method.

We build the microscopic systems by assigning a microscopic grid for every degree of freedom on the macroscopic grid. 
Since we use nodal basis functions, every degree of freedom corresponds to a single physical location in the macroscopic grid.
By allowing the microscopic systems to correspond with degrees of freedom, by integrating the microscopic finite element functions on the finite element domain, we obtain a finite element function on the macroscopic domain, for which we can use classical finite element techniques.

\dealii{} has no specific support for multiscale problems. However, we can build upon its structure to create new components that can deal with objects like multiscale functions and multiscale solutions.

For our multiscale implementation, we need not use a separate instance for each macroscopic degree of freedom.
Specifically, we assume the same microscopic grid and triangulation for each microscopic instance. This allows us to reuse and share microscopic data structures throughout the simulation.

\subsection{Manufactured system}
We test the quality and correctness of the scheme and its implementation by simulating a more general problem, for which we can manufacture solutions. We compute the solution of this problem on subsequently finer meshes, and check if convergence rates are according to expectations.
The manufactured problem ($P_M$) is defined as follows:
\begin{align*}
    &-A\Delta_x\pi=f(\pi,g(\rho)) + p(t,x)  &\mbox{ in }S\times\Omega,\\
    &\partial_t\rho-D\Delta_y\rho = q(t,x,y)  &\mbox{ in }S\times\Omega\times Y,\\
    &D\nabla_y\rho\cdot n_y= k(\pi+p_F-R\rho) + r(t,x,y)&\mbox{ in } S\times\Omega\times\Gamma_R,\\
    &D\nabla_y\rho\cdot n_y=s(t,x,y)&\mbox{ in }S\times\Omega\times\Gamma_N,\\
    &\pi=u(t,x) &\mbox{ in }S\times\partial\Omega,\\
    &\rho(t=0,x,y)=\rho_I(x,y)&\mbox{ in } \overline{\Omega\times Y},
\end{align*}
where, in our test scenario, $p,q,r,s,u$ are chosen such that they lead to a $\pi$ and $\rho$ of which we know the explicit form. Note that if we let $p=q=r=s=u=0$, $(P_M)$ reduces to $(P_1)$.
Additionally, $f$ is defined in accordance to the first example provided in Section~\ref{sec:preliminaries}:
\begin{equation}
    f(\pi,g(\rho)) := \theta \min \left( |\pi|,|\pi|^\alpha \right) \min \left( 1,|g(\rho)| \right),
\end{equation}
where $\theta$ is defined in Table~\ref{tab:parameter}, $\alpha=0.5$ and $g$ is defined as
\begin{equation}
    g(s):= \int_Y s(\cdot,y) dy.
\end{equation}
For the remainder of this discussion, we choose $\Omega = [-1,1]^2$, $Y= [-1,1]^2$, $\Gamma_R = \{1,-1\} \times [-1,1]$, and $\Gamma_N = [-1,1] \times \{1,-1\}$.
This results in a convex domain $\Omega$ of which the Poincar\'e constant $\mathcal{C}_\Omega$ can be computed exactly according to \cite{payne60} and has a value of $\mathcal{C}_\Omega = \frac{2\sqrt{2}}{\pi}$, where $\pi$ in this expression represents the mathematical constant.

\subsection{Time discretization}
We discretize the microscopic equation in time with an implicit Euler scheme, while we use a Picard-like iteration scheme for the macroscopic equation. The Picard-like iterations avoid approximating the nonlinear term $f$ via e.g. Newton's method, and the implicit Euler scheme ensures more stability in the microscopic equation without significant complications, since this equation is linear.

We discretize time domain $S$ with time steps $0 < t_1 < \ldots < t_{N_T}$. Let $\tau_n$ be the time step size in time step $n$. Then, given data from time step $n-1$, we compute the approximations in time step $n$ by solving:

\begin{align}
    &A \int_\Omega \nabla \pifem \cdot \nabla \phi dx = \int_\Omega f\left( \pifemprev, g \left( \rhofemprev \right) \right)+ p^H_{n-1} \phi dx,\\
    &\int_Y \frac{\rhofem - \rhofemprev}{\tau_n} \psi + D\nabla\rhofem  \cdot \nabla \psi dy = \int_Y  q^{H,h}_{n} \psi dy \\
    & \quad+ \int_{\Gamma_R} \left( \kappa \left(\pifem + p_F - R\rhofem \right) + r^{H,h}_{n}\right)\psi + \psi d\sigma_y + \int_{\Gamma_N} s_n^{H,h} \psi d\sigma_y.
\end{align}
We postpone proving the convergence of this scheme to a forthcoming publication.

Prior to solving this scheme, one first needs to determine for $\pi_I(x) = \pi(x,0)$ as a solution to \eqref{eq:init_pi}.
Choosing an initial guess rather than computing an actual value will create an error that propagates into future iterations of the scheme, because of the time-dependent nature of the equations.

We account for this by first solving the following discrete system
\begin{equation*}
    A \int_\Omega \nabla \pi^H_{I,k} \cdot \nabla \phi dx = \int_\Omega f\left( \pi^H_{I,k-1} g \left( \rho^{H,h}_I \right)\right)+ p^H_{0}\phi dx,
\end{equation*}
subject to the macroscopic boundary conditions for $\pi$.
Starting with an initial guess for $\pi^H_{I,0}$, we iterate until $\norm{\pi^H_{I,k} - \pi^H_{I,k-1}}_\MLL < \varepsilon$ for some threshold $\varepsilon$.
Then, $\pi^H_{I,k}$ is chosen as the initial value $\pi^H_I$ for ($P_M$).

\subsection{Results}
We fix the parameters according to the values presented in Table~\ref{tab:parameter}.

\begin{table}[ht]
    \centering
    \begin{tabular}{|l|l|}
    \hline
    \textbf{Parameters} & \textbf{Values} \\ \hline
    $A$                 & 3               \\ \hline
    $D$                 & 1               \\ \hline
    $\theta$            & 0.25             \\ \hline
    $\kappa$            & 1               \\ \hline
    $p_F$               & 4               \\ \hline
    $R$                 & 2               \\ \hline
    \end{tabular}
    \caption{Parameter values used in the simulation.}
    \label{tab:parameter}
\end{table}
We approximate the solutions to ($P_M$) with piece-wise linear basis functions, and compute the cell contributions with a third order Gaussian quadrature rule.
We test the implementation by solving ($P_M$), choosing $p,q,r,s$ such that solutions $\pi,\rho$ become:
\begin{equation}
    \begin{split}
        \pi(x,t) &= \cos\left(2x_0e^{-Dt} \sqrt{ \frac{\theta}{A}} \right) + \cos\left(2x_1e^{-Dt} \sqrt{ \frac{\theta}{A}} \right)\\
        \rho(x,y,t) &= e^{-2Dt}\left(\cos \left( y_0 \right) \cos \left( y_1 \right) + \sin(x_0) + \sin(x_1) + 6\right)\
    \end{split}
\end{equation}

Running this simulation for increasingly smaller $\tau_n$ (to account for the time discretization error), $h$ and $H$, yields the errors presented in Table~\ref{tab:macro_convergence} and Table~\ref{tab:micro_convergence}. The values $p_i$ and $q_i$ for $i=1,2$ represent the (subsequent) observed order of convergence of the finite element error and the gradient error, respectively. For example: refining the macroscopic grid results in an observed error of size $\norm{e_\pi}_\MLL = \bigo{H^{p_1}}$.

\begin{table}[ht]
\centering
\begin{tabular}{c|c|c|c|c|c|c}
\toprule
 MDoFs & $H$ & $\tau_n$  &$\norm{e_\pi}_\MLL$ & $\norm{\nabla_x e_\pi}_\MLL$ & $p_1$ & $q_1$ \\
\midrule
    81 & 2.500e-01 &    0.2500 & 2.329e-03 & 2.542e-02 &       - &    - \\
   144 & 1.818e-01 &    0.1250 & 1.421e-03 & 1.823e-02 & 1.551 & 1.044 \\
   289 & 1.250e-01 &    0.0625 & 6.114e-04 & 1.243e-02 & 2.251 & 1.022 \\
   576 & 8.696e-02 &    0.0312 & 2.961e-04 & 8.601e-03 & 1.998 & 1.015 \\
  1089 & 6.250e-02 &    0.0156 & 1.537e-04 & 6.157e-03 & 1.985 & 1.012 \\
  2116 & 4.444e-02 &    0.0078 & 8.125e-05 & 4.369e-03 & 1.870 & 1.006 \\
  4225 & 3.125e-02 &    0.0039 & 3.886e-05 & 3.069e-03 & 2.094 & 1.003 \\
\bottomrule
\end{tabular}
\caption{Macroscopic error and convergence rates for the manufactured problems. $p$ represents the subsequent observed order of convergence of the finite element error, $q$ represents the subsequent observed order of convergence of the error of its gradient.}
\label{tab:macro_convergence}
\end{table}

\begin{table}[ht]
\centering
\begin{tabular}{c|c|c|c|c|c|c}
\toprule
    mDoFs & $h$ & $\tau_n$ &  $\norm{e_\rho}_\mLL$ &$\norm{\nabla_y e_\rho}_\mLL$ & $p_2$ & $q_2$\\
\midrule
     6561 & 2.500e-01 &    0.2500 & 2.743e-01 & 3.560e-01 &       - &      - \\
    20736 & 1.818e-01 &    0.1250 & 1.486e-01 & 2.000e-01 & 1.925 & 1.811 \\
    83521 & 1.250e-01 &    0.0625 & 7.759e-02 & 1.103e-01 & 1.734 & 1.588 \\
   331776 & 8.696e-02 &    0.0312 & 3.968e-02 & 6.198e-02 & 1.848 & 1.588 \\
  1185921 & 6.250e-02 &    0.0156 & 2.006e-02 & 3.686e-02 & 2.066 & 1.574 \\
  4477456 & 4.444e-02 &    0.0078 & 1.008e-02 & 2.307e-02 & 2.019 & 1.374 \\
 17850625 & 3.125e-02 &    0.0039 & 5.057e-03 & 1.501e-02 & 1.958 & 1.220 \\
\bottomrule
\end{tabular}
\caption{Microscopic error and convergence rates for the manufactured problems. The microscopic degrees of freedom (mDoFs) are summed over all microscopic grids. $p$ represents the subsequent observed order of convergence of the finite element error, $q$ represents the subsequent observed order of convergence of the error of its gradient.}
\label{tab:micro_convergence}
\end{table}

Ignoring the time discretization as an error source, we observe the finite element error behaves according to the theory:
\begin{equation}
    \norm{e_\pi}_\MLL + \norm{e_\rho}_\mLL = C \left(h^2 + H^2\right) + \textrm{h.o.t.},
\end{equation}
thereby confirming Theorem~\ref{thm:priori}.
Noteworthy is the fact that the microscopic error accounts for a large part in the total error. This might be due to the fact that the macroscopic equation lacks a time derivative.

For $t=0.5$, a representation of the macroscopic solution and a graphical representation of the microscopic solutions are represented in Figure~\ref{fig:macro_plot} and Figure~\ref{fig:patched}, respectively.

\begin{figure}[ht]
    \centering
    \includegraphics[width=0.8\textwidth]{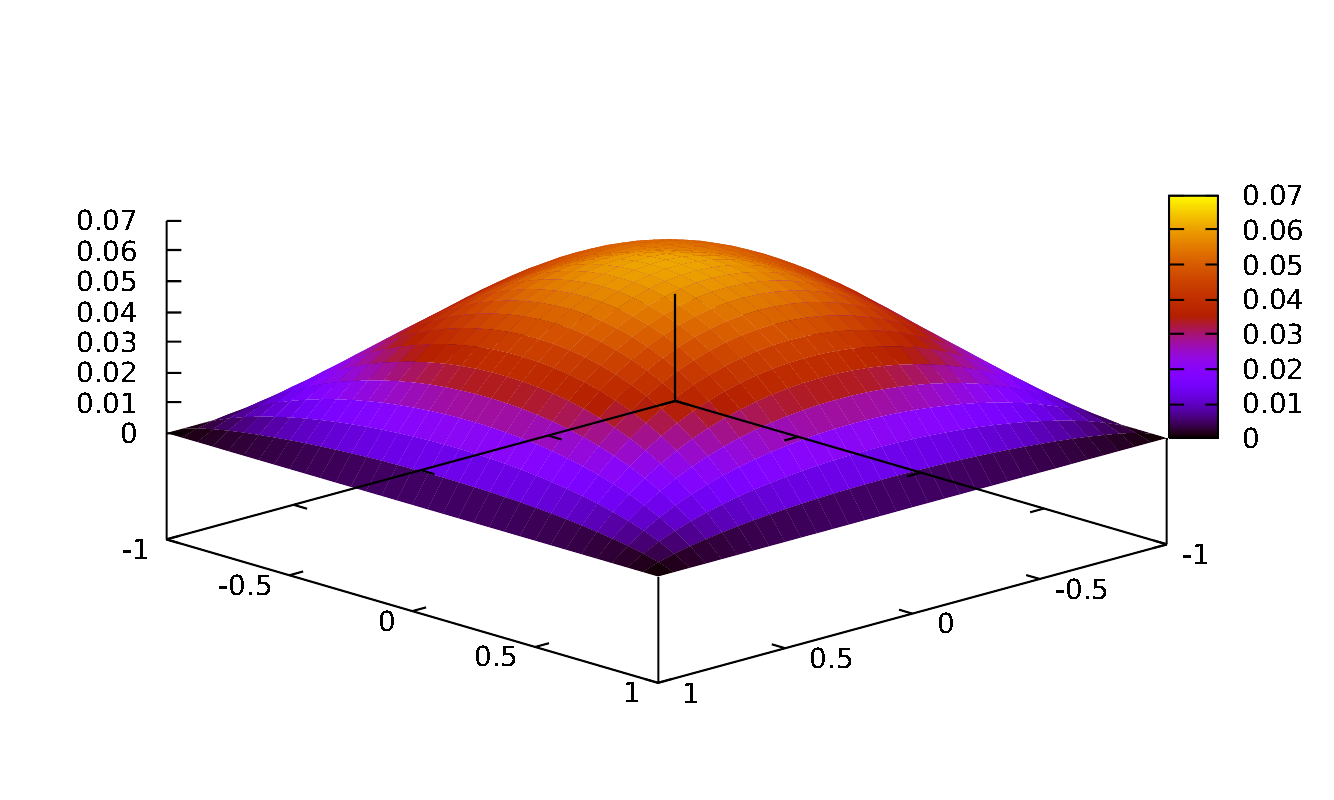}
    \caption{A snapshot of the macroscopic solution of $P_1$ for $t=0.5$.}
    \label{fig:macro_plot}
\end{figure}

\begin{figure}[ht]
    \centering
    \includegraphics[width=0.8\textwidth]{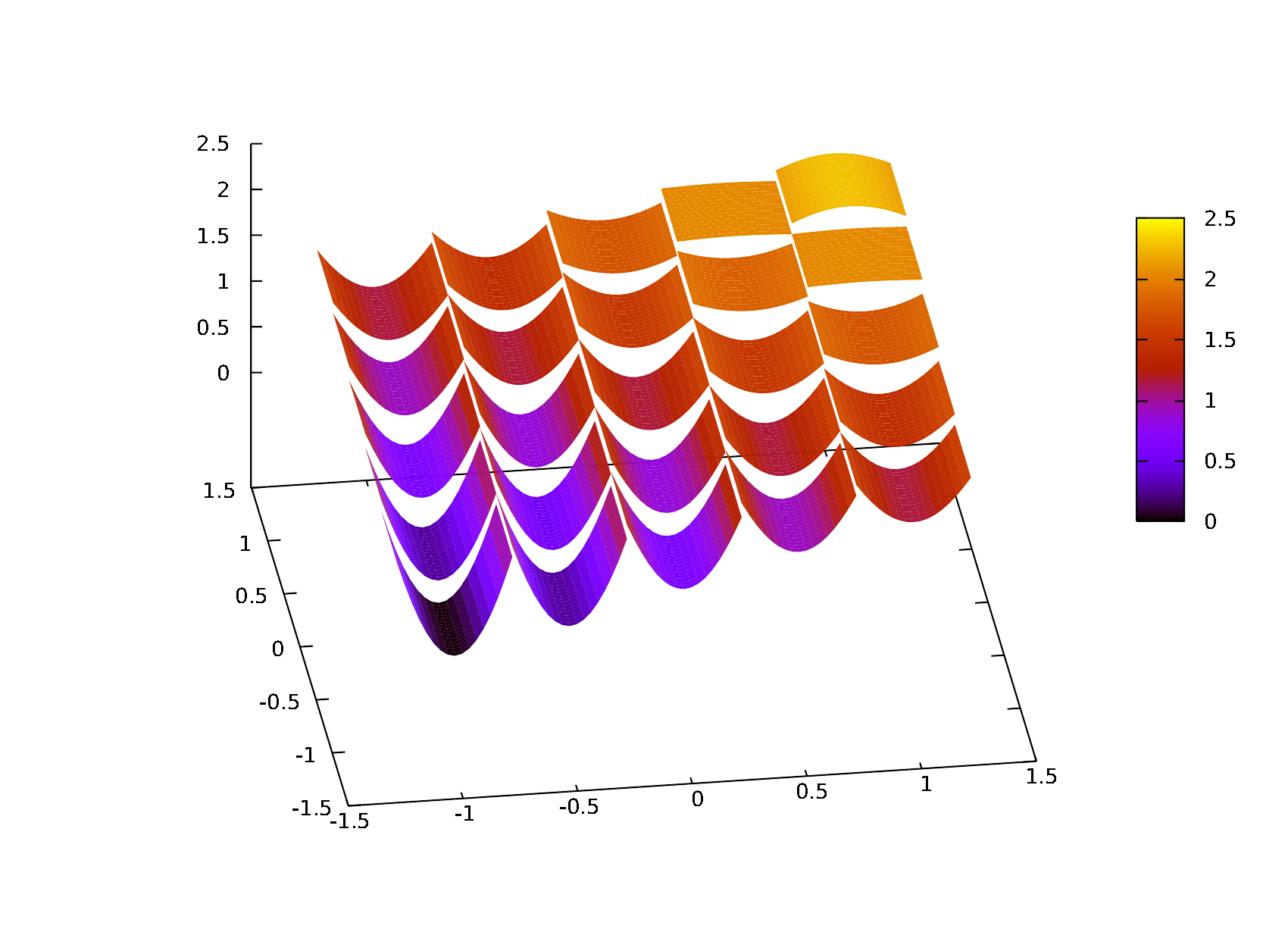}
    \caption{A collection of microscopic solutions at their respective locations in a coarse macroscopic grid, $t=0.5$.}
    \label{fig:patched}
\end{figure}

 \section{Conclusion and future work}
\label{sec:conclusion}
 We constructed a semidiscrete Galerkin approximation of our semi-linear elliptic-parabolic two scale system and  showed that this  approximation is well-posed. Furthermore,  the obtained sequence of Galerkin approximants based on finite elements converges in suitable spaces to the weak solution to the continuous system.  Under additional regularity assumptions, we derived  \textit{a priori} rates of convergence of our approximation to the target weak solution.
Finally, we implemented the fully discrete system in \dealii{}, tested the convergence rates in practice, and observed the behaviour of the system for a certain set of parameters. We found convergence behaviour according to our analysis. Using this setup, we are now able to deal with truly scale-separated, two-scale problems, using a method fitting in the HMM framework.

As a natural next step, in a forthcoming work a numerical analysis study will address the  quality of the fully discrete two-scale Galerkin approximation as well as improved  numerical implementations of the method so that the proven convergence rates can be confirmed and variants of macroscopic refinement strategies can be tested, possibly a two-scale elliptic-elliptic PDE system (obtained by letting $t \to \infty$  in our elliptic-parabolic formulation).

\section*{Acknowledgements}

The authors acknowledge fruitful discussions with Prof.~M.~Asadzadeh (Chalmers University, Gothenburg, Sweden). OR thanks Dr.~D.~Tagami (Kyushu University, Japan) for valuable feedback and acknowledges partial support from Kungl.~Vetenskapsakademien, Sweden. ML, AM and OR thank Dr.~O.~Lakkis and Dr.~C.~Venkataraman (both with University of Sussex, UK) for the intensive interactions during the Hausdorff Trimester Program ``Multiscale Problems: Algorithms, Numerical Analysis and Computation'' (Bonn, January 2017) and discussions at the University of Sussex.
 \bibliography{literature}
 \bibliographystyle{plain}
\end{document}